\documentclass[a4paper,12pt,leqno,twoside]{article}

\usepackage[english]{babel}

\usepackage{amsfonts}
\usepackage{amssymb}
\usepackage{amsthm, amsmath}
\usepackage{eucal}
\usepackage{mathrsfs}
\allowdisplaybreaks[0]
\usepackage[hypertexnames=false]{hyperref}
\usepackage{bbm}
\usepackage{graphicx}
\usepackage{caption}
\usepackage{subcaption}

\newcommand\testopari{\sc\small L. Calatroni, B. D\"{u}ring and C.-B. Sch\"{o}nlieb}
\newcommand\testodispari{\sc\small ADI schemes for a fourth-order nonlinear PDE}
\usepackage[usenames,dvipsnames]{color}

\newcommand{\R}{\mathbb{R}}

\newtheorem{theorem}{Theorem}[section]

\newtheorem{definition}[theorem]{Definition}

\providecommand{\norm}[1]{\left\lVert#1\right\rVert}

\begin{document}

\thispagestyle{empty}

\begin{center}
{\large \bf ADI SPLITTING SCHEMES FOR A FOURTH-ORDER \\[0.1cm]
NONLINEAR PARTIAL DIFFERENTIAL EQUATION \\[0.2cm]
 FROM IMAGE PROCESSING} \\[1cm]
\end{center}

\begin{center}
                  {\sc Luca Calatroni}\\
 Cambridge Centre for Analysis, University of Cambridge\\
 Wilberforce Road, CB3 0WA, Cambridge, United Kingdom \\
 (lc524@cam.ac.uk) \\[0.4cm]
{\sc Bertram D\"{u}ring}\\
Department of Mathematics, University of Sussex\\
Pevensey II, BN1 9QH, Brighton, United Kingdom\\
 (b.during@sussex.ac.uk) \\[0.4cm]
 {\sc Carola-Bibiane Sch\"{o}nlieb} \\
 Department of Applied Mathematics and Theoretical Physics (DAMTP)\\
 University of Cambridge, Wilberforce Road, CB3 0WA, Cambridge, United Kingdom \\
 (cbs31@cam.ac.uk)
       \end{center}
\vskip0.5cm

\begin{center}
  Dedicated to Prof.\ Arieh Iserles, on the occasion of his 65th birthday.
\end{center}

\begin{abstract}

We present directional operator splitting schemes for the numerical solution of a fourth-order, nonlinear partial differential evolution equation which arises in image processing. This equation constitutes the $H^{-1}$-gradient flow of the total variation and represents a prototype of higher-order equations of similar type which are popular in imaging for denoising, deblurring and inpainting problems. The efficient numerical solution of this equation is very challenging due to the stiffness of most numerical schemes.
We show that the combination of directional splitting schemes with implicit time-stepping provides a stable and computationally cheap numerical realisation of the equation.

\end{abstract}

\centerline{--------------------------------------------------------------}
\centerline{--------------------------------------------------------------}


\section{Introduction} \label{sec:int}

In this paper we propose directional operator splitting methods for the numerical solution of fourth--order nonlinear partial differential equations of the type
\begin{equation} \label{4thpde}
\begin{aligned}
& u_t =  \nabla\cdot\left(h(u)\nabla q \right) \quad q\in \partial \mathcal{E}(u)  & & \textrm{ in } \Omega\times (0,\infty),\\
& u(t=0) = u_0 & & \textrm{ in } \Omega,
\end{aligned}
\end{equation}
where $\mathcal E$ is the total variation (TV) seminorm (see, for instance, \cite{ambrosio})
\begin{equation} \label{TVfunct}
\mathcal{E}(u):=|Du|(\Omega)=\sup_{\textbf{p}\in C^{\infty}_0(\Omega;\R^2),\ \norm{\textbf{p}}\leq 1}\int_\Omega u\nabla\cdot\textbf{p}\ dx.
\end{equation}
and $\partial\mathcal{E}{u}$ is the subdifferential of $\mathcal{E}$ in $u$ (see \cite{ambrosio}). Here $\Omega\subset\R^2$ is open and bounded with Lipschitz boundary, $h:\R\rightarrow\R$, and $u_0$ a sufficiently regular initial condition. In our considerations equation \eqref{4thpde} is endowed with periodic boundary conditions. The elements $q$ of the subdifferential $\partial \mathcal E$ have the property that, if $q\in \partial \mathcal{E}(u)$, then (see \cite[Proposition 4.1]{vese}):
\begin{equation}  \label{subdiffTVel}
q=\begin{cases}
-\nabla\cdot\left(\frac{\nabla u}{|\nabla u|}\right) & \textrm{if } |\nabla u(x)|\neq 0,\\
0 & \textrm{if } |\nabla u(x)| = 0.
\end{cases}
\end{equation}
The above characterization shows the fourth-differential order and the strong nonlinearity in equation \eqref{4thpde}.


Equations of the form \eqref{4thpde} typically arise in the study of tension driven flow of thin liquid films (see \cite{myers,oron}), and have recently found applications in image processing (see, e.g., \cite{osher,burg2,schon1}). In this paper we limit ourselves to the consideration of one prototype of \eqref{4thpde}, that is we consider the special case where $h(u)\equiv 1$. The equation we obtain in that case reads
\begin{equation} \label{tvH-1}
 u_t=\Delta q,\quad q\in \partial \mathcal{E}(u), \quad \text{ in }\Omega\times (0,\infty),
\end{equation}
which constitutes a gradient flow of the total variation seminorm in the space $H^{-1}$.

In image processing problems -- such as image denoising and inpainting -- nonlinear partial differential equations of the type $\eqref{tvH-1}$ became popular due to their ability of preserving edges in the process of reconstruction (see \cite{ROF,ChL,rudin} for instance). In the latter works the authors typically deal with second--order partial differential equations, that is $L^2$-gradient flows of the total variation functional. Instead, in \cite{Me,osher,MG,LV,BuDiWei} an approach like \eqref{tvH-1} is proposed for image denoising and image decomposition. When applied to image denoising, a given noisy image $u_0$ is decomposed into its piecewise smooth part $u=u(T)$, solution of \eqref{tvH-1} at time $T$, and its oscillatory, noisy part $n$, i.e., $u_0=u+n$. Similarly, in image decomposition the piecewise smooth part $u$ represents the structure/cartoon part of the image, and the oscillatory part $n$ its texture part. The advantage of using an $H^{-1}$ subgradient of the total variation rather than an $L^2$ subgradient is that \eqref{tvH-1} shows better separability of the data into oscillatory and piecewise constant components and it reduces unwanted artifacts of the total variation filter such as staircasing, cf. \cite{osher,LV,LT}.

In \cite{burg2,schon} equation \eqref{tvH-1} is used for image inpainting, i.e. the problem of filling the missing parts of damaged images using the information from the surrounding areas. Here, the higher-order subgradient is necessary for improving upon the ability of total variation-type inpainting to smoothly connect image structures also over large distances, cf. for instance \cite{chan1,chan2,CS,CKS,MMo,ES}. In \cite{burg2} the inpainted image $u$ is reconstructed from a given image $f\in L^2(\Omega)$ which is damaged inside the so-called inpainting domain $D\subset\Omega$ by evolving it via the flow
\begin{equation}  \label{tvh-1inp}
\begin{aligned}
& u_t=\Delta q,\quad q\in \partial TV(u), & & \textrm{ in } D,\\
& u = f & & \textrm{ in } \Omega\setminus D,
\end{aligned}
\end{equation}
where
\begin{equation*}
 TV(u):=\left\{\begin{aligned}
                 |Du|(\Omega)\quad&\text{if }\norm{u}_\infty\leq 1\quad\text{a.e. in }\Omega, \\
                 +\infty\quad&\text{otherwise}.
                 \end{aligned}\right.
\end{equation*}
Here, the $L^{\infty}$-bound on solutions of \eqref{tvh-1inp} is a technical assumption needed for the existence analysis in \cite{burg} and does not present a restriction when dealing with grayvalue images whose values are always bounded within a fixed interval.

If the function $h$ in \eqref{4thpde} is the identity function $h(u)=u$ we obtain the following equation
\begin{equation} \label{TVWassflow}
\begin{aligned}
& u_t=\nabla\cdot(u\nabla q), \quad q\in\partial \mathcal E(u),& & \textrm{ in } \Omega\times (0,\infty),\\
& u(0,x)=u_0(x)\geq 0\quad& &\text{in }\Omega.
\end{aligned}
\end{equation}
This equation can be formally derived as the $L^2$-Wasserstein gradient flow of the total variation $\mathcal E$ in \eqref{TVfunct} and
has been proposed in \cite{burg} for density estimation and smoothing. Equation \eqref{TVWassflow} has been further investigated in \cite{dur}, where the authors numerically studied the scale space properties and high-contrasting effects of the equation by exploiting a numerical scheme similar to the ones we consider in this paper. In \cite{BCDS} the authors consider such equation for the problem of denoising, solving an alternative primal-dual formulation of \eqref{TVWassflow} with a Newton-type method.


From a computational point of view, finding numerical schemes that solve equations of the form \eqref{4thpde} is a challenging problem. Dealing with an evolutionary nonlinear fourth-order partial differential equation, we would like to find an efficient and easily applicable method. Note that a naive explicit discretisation in time may restrict the choice of the time-step size $\Delta t$ up to an order $\Delta x^4$, where $\Delta x$ denotes the step size of the spatial grid, compare \cite{smereka}. Moreover, the strong nonlinearity of subgradients of the total variation adds additionally constraints to the stability condition of a discrete time stepping scheme, compare \cite{ChMu,BuDiWei}. In particular, when approximating the subgradient of the total variation by regularising it -- either by a square root $\varepsilon$-regularisation or by a regularisation of its dual formulation --  the size of the regularisation parameter encodes the accuracy of this approximation, that is the strength of the nonlinearity in the approximated subgradient. The presence of this nonlinearity together with the fourth differential order of the equation then may result in restrictive stability conditions on numerical time stepping for small values of this regularisation parameter.\\
Having these complications in mind, some numerical methods have been proposed to solve equations like \eqref{tvH-1}. Lieu and Vese \cite{LV} proposed a numerical method to solve TV-$H^{-1}$ denoising/decomposition by using the Fourier representation of the $H^{-1}$ norm on the whole $\mathbb{R}^d$, $d\geq 1$. This leads to a second-order PDE defined in Fourier space. In \cite{elliott1} and \cite{elliott2} the authors propose an algorithm using a finite element method to solve such equation, while in \cite{AuCh,schon2} a dual approach similar to the one described in \cite{chamb} is presented with interesting applications both to denoising and inpainting. In \cite{schon} the authors present results of convergence and unconditional stability for a particular numerical splitting method solving equation \eqref{tvH-1}, called \emph{convexity splitting}. Therein, the equation is modeled as the gradient flow in $H^{-1}$ of the difference of two convex energies. The result is the presence of a linear diffusion term in the numerical scheme which balances the unstable behaviours coming from the nonlinear terms.




In this work we are interested in performing a \emph{directional\/}  operator splitting for the numerical solution of equation \eqref{tvH-1}, i.e.\  \eqref{tvh-1inp} for image inpainting. In particular, we consider \emph{alternating directional splitting (ADI)} methods that reach back to \cite{peaceman1} and have been further developed in, e.g, \cite{hunds1,hunds,barash,hout1,hout2,wit}. We discuss three different splitting methods and their stability properties. Our investigation culminates in the proposal of a fully-implicit, additive multiplicative operator splitting scheme for a slightly modified form of equation \eqref{tvH-1}. In our numerical experiments we obtain stable numerical solutions for arbitrary values of the regularization parameter, suggesting unconditional stability of the method. This allows to use fewer iterations (larger time steps) while each iteration is still computationally cheap due to the directional splitting. \\
We want to stress, that in the existing literature so far, most methods have been considered for the solution of second-order linear and nonlinear PDEs. Up to our knowledge the following discussion is the first one that carries out an ADI splitting for a higher-order total variation flow, which is both highly nonlinear and of fourth differential order.


\smallskip

\textbf{Outline of the paper.}
After the preliminary Section \ref{sec:prel}, in Section~\ref{sec:ADI} we will review the general features
of the main ADI methods. We will apply them to equation \eqref{tvH-1} in
Section~\ref{sec:ADIappl}. In order to come up with new and efficient
ADI methods solving numerically \eqref{tvH-1}, we exploit at first
(see Section~\ref{sec:bih}) an auxiliary, linear, fourth-order
equation. We next present in Section~\ref{sec:prim} our ideas applied
to the nonlinear cases we are interested on. Our strategy is as
follows: after linearising the equation in a suitable way, we expand
and discretise the appearing differential operators and we
build up the ADI schemes where we alternatively apply differential
operators acting just along either the $x$- or the $y$-direction.
According to the different choices of linearisation
performed, some mixed derivative terms may appear as well. They need
to be controlled properly in order to get stable results. Stability
properties appear to be related also with the regularisation of the
term $|\nabla u|$ appearing in both equations. In Sections
\ref{sec:lin1} and \ref{sec:lin2} we will give more details on that,
presenting in Section~\ref{sec:AMOS} a modified stable ADI scheme
dealing implicitly just with pure
derivatives. Section~\ref{sec:primdual} presents a quite different
approach to solve the so-called \emph{primal-dual} formulation of the TV-$H^{-1}$ problem.
In the concluding Section~\ref{sec:numres} we present some numerical results obtained
by applying the ADI methods above to some problems arising in image processing.

\section{Preliminaries} \label{sec:prel}

In this section we briefly introduce the main notation employed in the paper and the definition of the differential discrete operators we will use in the following to describe the numerical schemes. Dealing with a fourth-order nonlinear PDE, we also need to make precise what we mean exactly by \emph{stability} of a numerical scheme.

\subsection{Notation}

In this paper we discuss the numerical solution of evolutionary differential equations. We need to distinguish between the exact solution $u$ of the continuous equation and the approximate semi-discrete and fully discrete solutions of the numerical schemes. Let $u(x,y,t)$ the exact solution at point $(x,y)\in\Omega=[a,b]\times[c,d]$ and $t\geq 0$. We write $u_n(x,y)$ to indicate the approximation in time of the function $u(x,y,t_n)$, where $(x,y)\in\Omega,\ t_n=n\Delta t,\ n\geq 0$ and $\Delta t$ is the time step size. Further, we approximate $\Omega$  by a finite grid $\{a=x_1<\ldots<x_N=b\}\times\{c=y_1<\ldots<y_M=d\}$ with equidistant spatial size $h=(b-a)/N=(d-c)/M$. We then denote by $U(t)$ the approximation of $u(x_i,y_j,t)$ in the node $(x_i,y_j)$ at time $t$. Finally, the fully discrete approximation of $u$ is denoted by $U=(U_n)_{i,j}$. When dealing with vectors, we will indicate their components using the superscripts notation: $\textbf{Y}=\left( Y^1\, Y^2\right)^{\top}$. Finally, we will indicate by $\delta_{*}$ the differential operator acting in the direction $*$ and we will use the notations $D_{\nabla}$, $D_{\hbox{div}}$ and $D_{\Delta}U$ to indicate the discrete gradient, divergence and laplacian of the approximating solution $U$, respectively.

\subsection{Definition of the discrete operators}
In the following we define all the discrete operators that we use for approximating the spatial derivatives that appear in our numerical schemes. We use periodic boundary conditions throughout. We also performed the discretisation changing the boundary conditions to homogeneous Neumann boundary conditions, but this did not affect the numerical results significantly. 

We approximate the first derivatives $u_x(x_i,y_j,\cdot)$ by forward differences $(\delta_{x}^+ U)_{i,j}$ and backward differences $(\delta_{x}^- U)_{i,j}$, where:
\begin{equation*}
(\delta_{x}^+ U)_{i,j}=\frac{u_{i+1,j}-u_{i,j}}{h}
   \qquad\hbox{and }\qquad
(\delta_{x}^- U)_{i,j}=\frac{u_{i,j}-u_{i-1,j}}{h}
\end{equation*}
The first derivatives of $u$ with respect to $y$ are approximated analogously by $(\delta_{y}^+ U)_{i,j}$ and by $(\delta_{y}^- U)_{i,j}$. The pure second derivatives will be approximated by using the five-point formula, this means the Laplacian operator $\Delta u=u_{xx}+u_{yy}$ is approximated by:
\begin{equation*}
(D_{\Delta} U)_{i,j}=(\delta_{xx} U)_{i,j}+(\delta_{yy} U)_{i,j}=\frac{u_{i+1,j}+u_{i-1,j}+u_{i,j+1}+u_{i,j-1}-4u_{i,j}}{h^2}.
\end{equation*}
The mixed derivatives $u_{xy}(x_i,y_j,\cdot)$ are approximated by:
\begin{equation*}
(\delta_{xy} U)_{i,j}=\frac{u_{i+1,j+1}+u_{i-1,j-1}-u_{i-1,j+1}+u_{i+1,j-1}}{4h^2}.
\end{equation*}
For the discretisation of the pure and mixed fourth-order derivatives appearing in Section \ref{sec:bih}, we will use a $5 \times 5$ stencil and approximate by:
\begin{align}
(\delta_{xxxx}U)_{i,j}=(\delta_{xx}(\delta_{xx}U))_{i,j}&=\frac{u_{i+2,j}-4u_{i+1,j}+6u_{i,j}-4u_{i-1,j}+u_{i-2,j}}{h^4}, \notag \\
(\delta_{yyyy}U)_{i,j}=(\delta_{yy}(\delta_{yy}U))_{i,j}&=\frac{u_{i,j+2}-4u_{i,j+1}+6u_{i,j}-4u_{i,j-1}+u_{i,j-2}}{h^4}, \notag \\
(\delta_{xxyy}U)_{i,j}=(\delta_{xy}(\delta_{xy}U))_{i,j}=&\frac{u_{i+2,j+2}+u_{i-2,j-2}+4u_{i,j}+u_{i-2,j+2}+u_{i+2,j-2}}{16h^4}\notag\\
&-\frac{u_{i,j+2}+u_{i+2,j}+u_{i-2,j}+u_{i,j-2}}{8h^4}.\notag
\end{align}

\subsection{Definition of stability}
In the following, we introduce the notion of \emph{stability} of the numerical schemes solving the nonlinear equations \eqref{tvH-1} and \eqref{tvh-1inp}:
\begin{definition} \label{def_stab}
Let $u$ be an element of a suitable function space $\mathcal{H}$ defined on $\Omega\times[0,T]$, with $\Omega\subset\R^2$ open and bounded and $T>0$. Let $F$ be a real valued function and $u_t=G(u,D^\alpha u)$ a partial differential equation with all spatial $D^\alpha$ with order $|\alpha|\leq 4$. A corresponding time stepping method
\begin{equation} \label{stability}
u_{n+1}=u_n+\Delta t G_n(u_n,u_{n+1},D^\alpha u_n, D^\alpha u_{n+1})
\end{equation}
where $G_n$ is a suitable approximation of $G$ in $u_n$ and $u_{n+1}$ is \textbf{unconditionally stable} if all the solutions of \eqref{stability} are bounded for all $\Delta t>0$ and all $n$ such that $n\Delta t\leq T$ and \textbf{stable} if, for a given $\Delta t>0$, the solutions of \eqref{stability} are bounded for all $n$ such that $n\Delta t\leq T$.
\end{definition}


\section{ADI splitting schemes}  \label{sec:ADI}
\setcounter{equation}{0}

In this section we introduce the numerical method of directional splitting. In particular, we discuss three types of directional splitting: the Peaceman-Rachford scheme, the Douglas-Hundsdorfer scheme, and an additive multiplicative operator splitting scheme.\\
In everything that follows we consider large systems of generic ordinary differential equations that arise from a
semi-discretisation of initial boundary value problems such as
\begin{equation} \label{ODEs}
\left\{\begin{aligned}
U'(t)&=F(t,U(t))\quad\text{for }t\geq 0, \\
U(0)&=U_0,
\end{aligned}\right.
\end{equation}
for a given function $F$. The explicit dependence on time of $F$ may not appear, thus considering \emph{autonomous} systems. Typically, directional splitting methods (see, for instance, \cite{hunds}) decompose the function $F$ into the sum:
\begin{equation} \label{decomp}
F(t,U)=F_0(t,U)+F_1(t,U)+\dots+F_s(t,U)\quad\text{for some }s\geq 1,
\end{equation}
where $s$ is the spatial dimension of the problem. The components $F_j$, $j=1,\ldots, s,$ encode the linear action of $F$ along the space direction $j=1,\ldots, s,$ respectively, and $F_0$ contains contributes coming from mixed directions and/or non stiff nonlinear terms. A generic ADI scheme constitutes a time-stepping method that treats the unidirectional components $F_j, j\geq 1$ implicitly and the $F_0$ component, if present, explicitly in time.

\subsection{The Peaceman-Rachford scheme}

The first method we consider in this framework is the second-order \emph{Peaceman-Rachford} ADI method (see \cite{peaceman1} and \cite{hunds}) where the operator $F$ can be splitted into $F=F_1+F_2$, i.e. no mixed derivative or nonlinear terms are present. Let $\Delta t>0$ be the time step size of the scheme, then for every $n\geq 1$ the numerical solution $U_{n+1}$ of \eqref{ODEs} is computed as follows:
\begin{equation}  \label{peacemanrach}
\left\{\begin{aligned}
& U_{n+1/2}=U_n+\frac{\Delta t}{2}F_1(t_n,U_n)+\frac{\Delta t}{2}F_2(t_{n+1/2},U_{n+1/2}), \\
& U_{n+1}=U_{n+1/2}+\frac{\Delta t}{2} F_1(t_{n+1},U_{n+1})+\frac{\Delta t}{2}F_2(t_{n+1/2},U_{n+1/2}).
\end{aligned}\right.
\end{equation}
In \eqref{peacemanrach} we can see that forward and backward Euler are applied alternatingly in a symmetrical fashion, thus obtaining second-order accuracy (see \cite{hunds}). In each half time step, one of the two dimensions is treated implicitly, whereas the other one is explicitly considered. Note that \eqref{peacemanrach} can be generalised to problems with nonlinearities $f(U)$ and
$$
U'(t)=F(t,U(t)) + f(U(t)),
$$
by adding $\frac{1}{2} f(U_n)$ and $\frac{1}{2} f(U_{n+1/2})$ to the
right hand side of the first and second equation in \eqref{peacemanrach}
respectively. The scheme
\eqref{peacemanrach}, however, does not have a natural extension for
more than two components $F_j$, so more general ADI methods have been
proposed in the literature.

\subsection{The Douglas-Hundsdorfer scheme} \label{subsec_HUNDS}

Another ADI method that allows a more general decomposition like the one in \eqref{decomp} is the so-called \emph{Douglas\/} method (see \cite{houwen}, \cite{hunds} and \cite{hunds1}). In that, the numerical approximation in each time step is computed by applying at first a forward Euler predictor and then it is stabilised by fractional $s$ steps where just the unidirectional components $F_j$ appear, weighted by a parameter $\theta\in[0,1]$. Its size controls the implicit/explicit character of such steps. In other words, the unidirectional operators are applied to the convex combination $\theta U_{n+1}+(1-\theta)U_n$, thus considering fully implicit steps for $\theta=1$, explicit ones for $\theta=0$ and a Crank-Nicolson type scheme when $\theta=1/2$. The consistency order in time of the scheme is equal to two whenever $F_0=0$ and $\theta=1/2$ and it is of order one otherwise. In many applications we need to consider, however, $F_0\neq 0$ (for instance, if we are considering contributions coming from mixed derivative operators) for any given $\theta$. Some extensions of the Douglas scheme have been proposed in order to overcome this drawback. The following scheme proposed in \cite{hunds1} by Hundsdorfer is an extension of the \emph{Douglas} method where a second stabilising parameter $\sigma>0$ appears:
\begin{equation} \label{ADIgen}
\left\{ \begin{aligned}
& Y_0=U_n+\Delta tF(t_n,U_n) \\
& Y_j=Y_{j-1}+\theta\Delta t(F_j(t_{n+1},Y_j)-F_j(t_n,U_n)),\quad j=1,2,\dots,s \\
& \tilde{Y}_0=Y_0+\sigma\Delta t(F(t_{n+1},Y_s)-F(t_n,U_n)) \\
& \tilde{Y}_j=\tilde{Y}_{j-1}+\theta\Delta t(F_j(t_{n+1},\tilde{Y}_j)-F_j(t_{n+1},Y_s)),\quad j=1,2,\dots,s \\
& U_{n+1}=\tilde{Y}_s.
\end{aligned}\right.
\end{equation}
The advantage of this extension is that for any given $\theta$ the
scheme \eqref{ADIgen} has time-consistency order two if $\sigma=1/2$ and one
otherwise, independently of $F_0$. The parameter $\theta$ is typically fixed to $\theta=1/2$. Its choice is discussed in \cite{hunds1}. Larger values of $\theta$ give stronger damping of implicit terms, lower values typically better accuracy.  Stability properties of this scheme when applied to
linear convection-diffusion equations with mixed derivative terms have
been investigated in \cite{hout1, hout2}. There, the preferable value for $\theta$ is $\theta=1/2+\sqrt{3}/6$. In \cite{wit} the authors combine the approach presented above with iterative methods for solving nonlinear systems.

\subsection{Additive multiplicative operator splitting schemes}

As pointed out above, in both schemes \eqref{peacemanrach} and \eqref{ADIgen} some explicit terms appear. These may affect stability properties of the methods and, generally, their accuracy. As described in \cite{barash}, splitting schemes as directional splitting methods belong to the family of \emph{multiplicative locally one-dimensional\/} (LOD) schemes. In their general semi-implicit form, when a splitting similar to \eqref{decomp} holds, they appear as:
\begin{equation} \label{multsplit}
\prod_{i=0}^s(I-\Delta t F_i)U_{n+1}=U_n,
\end{equation}
where, similarly as before, each operator $F_i,$ $1\leq i\leq s$, is acting just along the
$i$-th direction,  whereas $F_0$ encodes mixed contributions. As
pointed out above, it is generally difficult to deal with such an operator as the matrix $(I-\Delta t F_0)$ is,
generally, not tridiagonal. For this reason, the scheme \eqref{ADIgen}
deals explicitly with such a term. Analogously, explicit components
appear also when applying \eqref{peacemanrach} because of the
alternating application of forward and backward Euler. As we will
point out later on, these explicit contributions may create stability
problems in the methods we are going to present. Following the
strategy presented in \cite{barash}, our attempt is to modify
\eqref{peacemanrach} such that no explicit contributions appear. The
cost of such an operation will be reducing the accuracy of the method
to order one, against the second-order achieved with the classical Peaceman-Rachford method \eqref{peacemanrach}. In order to preserve such accuracy as well as the symmetry of the method, at each time step two calculations are performed:

\begin{equation} \label{amos1}
\left\{\begin{aligned}
& (I-\Delta t F_1)U_{n^*}=U_n  \\
& (I-\Delta t F_2)\tilde{U}_{n+1}=U_{n^*}
\end{aligned}\right.
\quad\mbox{ and }\quad
\left\{
\begin{aligned}
& (I-\Delta t F_2)U_{n^\star}=U_n  \\
& (I-\Delta t F_1)\bar{U}_{n+1}=U_{n^\star}
\end{aligned}\right.
\end{equation}
which, written in the same form as \eqref{peacemanrach} and \eqref{ADIgen}, read as:
\begin{equation*}
\left\{\begin{aligned}
& U_{n^*}=U_n+\Delta t F_1(U_{n^*})  \\
& \tilde{U}_{n+1}=U_{n^*}+\Delta F_2(\tilde{U}_{n+1})
\end{aligned}\right.
\quad\mbox{ and }\quad
\left\{
\begin{aligned}
& U_{n^\star}=U_n + \Delta t F_2(U_{n^\star}) \\
& \bar{U}_{n+1}=U_n+ \Delta t F_1(\bar{U}_{n+1})
\end{aligned}\right. .
\end{equation*}
To get the numerical solution $U_{n+1}$ we simply average:
\begin{equation} \label{amos2}
U_{n+1}=\frac{\tilde{U}_{n+1}+\bar{U}_{n+1}}{2},
\end{equation}
thus ensuring a symmetric splitting. Due to the nature of such a method, we refer to \eqref{amos1}-\eqref{amos2} as \emph{additive multiplicative operator splitting} (AMOS) ADI method. Note, that this scheme is identical to an earlier version of the  well-known  \emph{Strang} splitting (see, for instance, \cite[(1.12), p.329]{hunds}). For nonlinear problems, such a scheme is second-order accurate, in contrast to first-order accuracy of the classical \emph{Strang} splitting scheme (compare \cite{barash}). Furthermore, the scheme \eqref{amos1}-\eqref{amos2} has the advantage of allowing a parallel implementation, as suggested in \cite{LNT1,LNT2,Weickert}.

\section{Applications to higher-order PDEs} \label{sec:ADIappl}
\setcounter{equation}{0}

As we are dealing with the particular case of regular domains in
$\R^2$, we will consider in the following $s=2$. Thus, the $F_1$
component ($F_2$, respectively) will contain operators acting just
along the $x$-direction ($y$-direction, respectively). When appearing,
the term $F_0$ will deal, instead, with the mixed $xy$-direction. Our
aim is to adapt the ADI schemes \eqref{peacemanrach}, \eqref{ADIgen}
and \eqref{amos1}-\eqref{amos2} to regularised versions of
equation \eqref{tvH-1} in Sections \ref{sec:prim} and
\ref{sec:primdual}.

\subsection{A linear example: the biharmonic equation}  \label{sec:bih}
We start by considering the auxiliary linear, fourth-order parabolic \emph{biharmonic} equation:
 \begin{equation} \label{biharmcont}
 u_t=-\Delta^2u\quad\text{in }\Omega\times (0,T).
 \end{equation}
Such equation appears in many applied mathematical models such as the Cahn-Hilliard models describing phase transitions and separation in binary mixtures (see, for instance, \cite{elliott} and \cite{novick}). Some work on ADI schemes applied to the biharmonic equation already exists in literature. It dates back to \cite{conte} where the authors consider the equation to model vibrational modes for thin plates and it has been analysed also in \cite{wit,hout1} where linear stability results are proved as well. We are looking for the solution $U$ of the following semi-discretised version of \eqref{biharmcont}:
\begin{equation}\label{biharm}
U_t=F(U):=-(D_{\Delta^2}) U=-\delta_{xxxx}U-\delta_{yyyy}U-2\delta_{xxyy}U.
\end{equation}
According to the rules of ADI schemes, we decompose the function $F$ into the sum
\begin{equation*}
F(U)=F_0(U)+F_1(U)+F_2(U)m
\end{equation*}
where the components $F_i,$ $i=0, 1, 2,$ are defined by
\begin{equation*}
F_0(u):=-2\delta_{xxyy}U,\quad F_1(u):=-\delta_{xxxx}U,\quad F_2(u):=-\delta_{yyyy}U,
\end{equation*}
and the differential operators have been discretised as discussed in Section \ref{sec:prel}. With this choice, we can find for every $n$ the approximating solution $U_{n+1}$ of \eqref{biharmcont} using the \emph{Hundsdorfer} ADI scheme \eqref{ADIgen}.

Simplifying the problem by
splitting the fourth-order equation into a mathematically equivalent autonomous system of two
partial differential equations of order two produces the system:
\begin{equation} \label{biharm1}
\left\{\begin{aligned}
    & U_t=D_{\Delta} V=\delta_{xx}V+\delta_{yy}V=F(U,V), \\
    & V=-D_{\Delta} U=-\delta_{xx}U-\delta_{yy}U=G(U,V).
    \end{aligned}\right.
\end{equation}
Then, the \emph{Hundsdorfer} scheme applied to \eqref{biharm} can be equivalently written as a coupled ADI scheme for approximate solutions $(U_n,V_n)$ of \eqref{biharm1}. For positive parameters $\theta, \sigma$ this gives:
\begin{equation}  \label{ADI1}
\left\{ \begin{aligned}
& \left(\begin{array}{c} Y^2_0 \\ Y^1_0 \end{array}\right)=\left(\begin{array}{c}G(U_n,V_n) \\ U_n+\Delta t F(U_n,Y^2_0)\end{array} \right), \\
& \left(\begin{array}{c} Y^1_1 \\ Y^2_1 \end{array}\right)=\left(\begin{array}{c} Y^1_0 \\  0 \end{array}\right)+\left(\begin{array}{c} \theta\Delta t F_1(Y^1_1,Y^2_1) \\ G_1(Y^1_1,Y^2_1)-G_1(U_n,V_n)\end{array} \right), \\
& \left(\begin{array}{c} Y^1_2 \\ Y^2_2 \end{array}\right)= \left(\begin{array}{c} Y^1_1 \\  0 \end{array}\right)+ \left(\begin{array}{c} \theta\Delta t F_2(Y^1_2,Y^2_2) \\ G_2(Y^1_2,Y^2_2)-G_2(U_n,V_n)\end{array} \right),  \\
& \left(\begin{array}{c} \tilde Y^2_0 \\ \tilde Y^1_0 \end{array}\right)=\left(\begin{array}{c} G(Y^1_2,Y^2_2) \\
Y^1_0+\sigma\Delta t(F(Y^1_2,\tilde Y^2_0)-F(U_n,V_n)) \end{array}\right), \\
& \left(\begin{array}{c} \tilde Y^1_1 \\ \tilde Y^2_1 \end{array}\right)=\left(\begin{array}{c} \tilde Y^1_0 \\  0 \end{array}\right)+\left(\begin{array}{c} \theta\Delta t F_1(\tilde Y^1_1, \tilde Y^2_1) \\ G_1(\tilde Y^1_1,\tilde Y^2_1)-G_1(Y^1_2,Y^2_2)\end{array} \right),\\
& \left(\begin{array}{c} U_{n+1} \\ V_{n+1} \end{array}\right)=\left(\begin{array}{c} \tilde Y^1_1 \\ 0 \end{array}\right)+\left(\begin{array}{c} \theta\Delta t F_2(\tilde Y^1_2, \tilde Y^2_2) \\ G_2(\tilde Y^1_2,\tilde Y^2_2)-G_2(Y^1_2,Y^2_2)\end{array} \right)
\end{aligned}\right.
\end{equation}
where the functions $F, F_1, F_2$ and $G,G_1,G_2$ are given by:
\begin{align}
& \left(\begin{array}{c} F_1(U,V) \\ G_1(U,V) \end{array}\right)=\left(\begin{array}{cc} A_1 & B_1 \\ C_1 & D_1 \end{array} \right)\cdot \left(\begin{array}{c} U \\ V \end{array}\right)=
\left(\begin{array}{cc} 0 & \delta_{xx} \\ -\delta_{xx} & 0 \end{array}\right)\cdot \left(\begin{array}{c} U \\ V \end{array}\right), \notag \\
&\left(\begin{array}{c} F_2(U,V) \\ F_2(U,V) \end{array}\right)=\left(\begin{array}{cc} A_2 & B_2 \\ C_2 & D_2 \end{array} \right)\cdot \left(\begin{array}{c} U \\ V \end{array}\right)=
\left(\begin{array}{cc} 0 & \delta_{yy} \\ -\delta_{yy} & 0 \end{array}\right)\cdot \left(\begin{array}{c} U \\ V \end{array}\right) \label{ADIbiharm}, \\
& F(U,V)=F_1(U,V)+F_2(U,V), \notag \quad G(U,V)=G_1(U,V)+G_2(U,V). \notag
\end{align}

To verify that \eqref{ADI1} gives the same solution as the \emph{Hundsdorfer} scheme applied to \eqref{biharm} we consider as example
the first implicit step in \eqref{ADI1} providing the approximations $(Y^1_1,Y^2_1)$. For the approximation $Y^1_1$ of $U_{n+1}$ we have:
\begin{equation*}
   Y^1_1=Y^1_0+\theta\Delta tF_1(Y^1_1,Y^2_1)=Y^1_0+\theta\Delta t\delta_{xx}(Y^2_1).
\end{equation*}
Using now the expression of $Y^2_1$ given by the implicit step relating to the equation for $V$ we have:
\begin{equation*}
 Y^1_1=Y^1_0+\theta\Delta t(\delta_{xx}(-\delta_{xx}(Y^1_1)+\delta_{xx}(U_n)))=Y^1_0+\theta\Delta t(-\delta_{xxxx}(Y^1_1)+\delta_{xxxx}(U_n)),
\end{equation*}
which, compared to the respective step performed applying the \emph{Hundsdorfer} scheme \eqref{ADIgen} to equation \eqref{biharm}, gives exactly the same result. We perform the same technique for the other implicit steps. Moreover, as the reader may note, in both the explicit steps of the scheme above we swapped the order of application of the method for consistency issues. Namely, we first find consistent approximations for $V_{n+1}$ using them to get consistent approximations of $U_{n+1}$. In the application of both these methods we get stable solutions both for $\Delta t=C(\Delta x)^3$ and also for $\Delta t=C(\Delta x)^2$, improving largely upon the condition $\Delta t  = C(\Delta x)^4$ that a naive explicit discretisation would give. Scheme \eqref{ADI1} is indeed \emph{unconditionally} stable, as proved earlier in \cite{hout1} and confirmed in our numerical experiments in Section \ref{sec:numres}. However, note that considering bigger time steps such as $\Delta t=C(\Delta x)$ the numerical accuracy suffers not producing sensible solutions.

\subsection{Applications to the primal formulation of TV-$H^{-1}$ equation}  \label{sec:prim}

We now aim to derive an ADI method solving the TV-$H^{-1}$ equation \eqref{tvH-1}. Expanding directly the differential operators appearing in the equation 
generates an intractable number of nonlinear terms of various differential orders. This makes a direct application of the ADI scheme to \eqref{tvH-1} impractical. Therefore, following the ideas presented in Section ~\ref{sec:bih}, we reduce the original fourth-order equation to an autonomous system of two second-order equations. In the following we consider the \emph{primal} formulation of \eqref{tvH-1}, in contrast to the \emph{primal-dual} formulation presented in Section \ref{sec:primdual}. To do so, we first regularise the subgradient of the total variation in \eqref{tvH-1} characterised by \eqref{subdiffTVel}. We use the standard \emph{square root $\varepsilon$-regularisation} that replaces $|\nabla u|$ by $|\nabla u|_\varepsilon:=\sqrt{|\nabla u|^2+\varepsilon}, 1\gg\varepsilon>0$, see for instance \cite{chan}. This results in the following regularised version of \eqref{tvH-1}
\begin{equation} \label{systemlin1}
u_t=-\Delta\nabla\cdot\left(\frac{\nabla u}{|\nabla u|_\varepsilon}\right)\quad\text{ and, equivalently, }\quad\left\{\begin{aligned}
& u_t=\Delta v, \\
& v=-\nabla\cdot\left(\frac{\nabla u}{|\nabla u|}_\varepsilon\right).
\end{aligned}\right.
\end{equation}
 In the following we present two different linearisations of the
 problem above. Heuristically, such a choice is important
 from two different points of view, intrinsically related to each
 other. The former is the accuracy of the scheme we are considering:
 rough linearisations (i.e.\ linearisations which consider most of the
 nonlinear terms explicitly evaluating them in one or more given
 approximations of the solution in previous time steps) are likely to
 present poor accuracy as well as stability issues. This is a
 general consideration in the numerical solution of every partial
 differential equation and it has to be taken into account and
 balanced with the choice of linearisation which might be more
 accurate and precise, but which could present, on the other hand,
 difficulties in its implementation and application. The latter
 point of view is, in some sense, peculiar to our choice of performing
 a directional splitting scheme. As pointed out above,
 our purpose is splitting our partial differential operator
 into the sum of components which are considered both explicitly (see
 $F_0$ above) and implicitly (see $F_1$ and $F_2$), as in \eqref{decomp}. As we are going to
 present in the following, the choice of the linearisation affects
 such a splitting as the $F_0$ component and the linearised quantities
 multiplying the differential operators acting in $x$ and $y$ may
 change accordingly. For instance, the $F_0$ component might not appear changing the choice of the ADI scheme we want to use.

\medskip

In the following we proceed by presenting two ADI schemes of the form
\eqref{peacemanrach} and \eqref{ADIgen} for two different
linearisations of \eqref{systemlin1}.
For a given initial condition $(U_0,V_0)$, our problem consists in
finding an approximation $(U_{n+1},V_{n+1})$ of the solution
$(u(t_{n+1}),v(t_{n+1}))$ to \eqref{systemlin1} for every $n\geq 0$. In the following we always linearise around the solution at the previous time step $(U_{n}, V_{n})$.

\subsubsection{The first linearisation} \label{sec:lin1}

Indicating by $\tilde U$ the value of the solution $U_n$ in the previous time step, our first choice of linearisation is the following (compare with \cite{dur}):
\begin{equation} \label{timestep1}
\left\{\begin{aligned}
& U_t=D_{\Delta} V, \\
& V=-\frac{\varepsilon+(\delta^+_y \tilde U)^2}{|D^+_{\nabla}\tilde U|^3_\varepsilon}\delta_{xx}U-\frac{\varepsilon+(\delta^+_x \tilde U)^2}{|D^+_{\nabla}\tilde U|^3_\varepsilon}\delta_{yy}U+2\frac{\delta^+_x \tilde U\delta^+_y \tilde U}{|D^+_{\nabla}\tilde U|^3_\varepsilon}\delta_{xy}U,
\end{aligned}\right.
\end{equation}
where $(U,V)$ is the semi-discrete approximation to a solution of \eqref{systemlin1}. Here, the linearisation of $V$ constitutes a semi-implicit approximation of the second-order nonlinear diffusion, evaluating all the first-order derivatives of the expansion in the previous time step.
As before, we use the following notation for the system \eqref{timestep1}:
\begin{equation}  \label{matrform}
\left(\begin{array}{c} U_t \\ V\end{array} \right)=\left(\begin{array}{c} F(U,V)\\ G(U,V)\end{array}\right) = \left(\begin{array}{cc} A & B\\ C & D\end{array}\right)\cdot \left(\begin{array}{c} U \\ V\end{array}\right)
\end{equation}
for suitable matrices $A,B,C$ and $D$ in $\R^{NM\times NM}$. We split $F$ and $G$ into the sum of three different terms: $F_0$ and $G_0$ containing the mixed derivative term and $F_1, G_1$ and $F_2, G_2$ containing the derivatives with respect to $x$ and $y$ only, respectively. This produces the splitting:
\begin{align}
& F(U,V)=F_0(U,V)+F_1(U,V)+F_2(U,V), \notag \\
& G(U,V)=G_0(U,V)+G_1(U,V)+G_2(U,V), \label{ADITV}
\end{align}
with:
\begin{align}
& \left(\begin{array}{c} F_0(U,V) \\ G_0(U,V) \end{array}\right)=\left(\begin{array}{cc} A_0 & B_0 \\ C_0 & D_0 \end{array} \right)\cdot \left(\begin{array}{c} U \\ V \end{array}\right)= \left(\begin{array}{cc} 0  & 0 \\ 2\frac{\delta^+_x\tilde U\delta^+_y\tilde U}{|D^+_{\nabla}\tilde U|^3_\varepsilon}\delta_{xy} & 0 \end{array}\right)\cdot \left(\begin{array}{c} U \\ V \end{array}\right), \notag \\
& \left(\begin{array}{c} F_1(U,V) \\ G_1(U,V) \end{array}\right)=\left(\begin{array}{cc} A_1 & B_1 \\ C_1 & D_1 \end{array} \right)\cdot \left(\begin{array}{c} U \\ V \end{array}\right)=
\left(\begin{array}{cc} 0 & \delta_{xx} \\ -\frac{\varepsilon+(\delta^+_y \tilde U)^2}{|D^+_{\nabla}\tilde U|^3_\varepsilon}\delta_{xx} & 0 \end{array}\right)\cdot \left(\begin{array}{c} U \\ V \end{array}\right), \label{ADITV2}\\
& \left(\begin{array}{c} F_2(U,V) \\ G_2(U,V) \end{array}\right)=\left(\begin{array}{cc} A_2 & B_2 \\ C_2 & D_2 \end{array} \right)\cdot \left(\begin{array}{c} U \\ V \end{array}\right)=
\left(\begin{array}{cc} 0 & \delta_{yy} \\ -\frac{\varepsilon+(\delta^+_x \tilde U)^2}{|D^+_{\nabla}\tilde U|^3_\varepsilon}\delta_{yy} & 0 \end{array}\right)\cdot \left(\begin{array}{c} U \\ V \end{array}\right). \notag
\end{align}
Due to the presence of a mixed derivative operator, a simple \emph{Peaceman-Rachford} scheme cannot be applied. The \emph{Hundsdorfer} scheme \eqref{ADIgen} is needed. The ADI scheme we use has the same form as for the biharmonic equation, i.e.\ we employ \eqref{ADI1} with $F, F_i$ and $G, G_i$ given by \eqref{ADITV}, \eqref{ADITV2}.

\subsubsection{The second linearisation} \label{sec:lin2}

Another possibility is to linearise the system \eqref{systemlin1} in the following way:
\begin{equation} \label{timestep2}
\left\{\begin{aligned}
U_t=&D_{\Delta} V, \\
 V=&-D^-_{\hbox{div}}\left(\frac{D^+_{\nabla} U}{|D^+_{\nabla} \tilde U|}_\varepsilon\right)=-\frac{1}{|D_{\nabla}^+ \tilde U|_\varepsilon}\delta_{xx}U+\frac{\delta_{x}^+ \tilde U\delta_{xx}\tilde U+\delta_{y}^+\tilde U\delta_{x}^-\delta^+_{y}\tilde U}{|D_{\nabla}^+\tilde U|^3_\varepsilon}\delta_{x}^{up}U \\
&\qquad \qquad \qquad \qquad \qquad \;-\frac{1}{|D_{\nabla}^+\tilde U|_\varepsilon}\delta_{yy}U+\frac{\delta_{x}^+\tilde U\delta_{x}^+\delta_{y}^-\tilde U+\delta_{y}^+\tilde U\delta_{yy}\tilde U}{|D_{\nabla}^+ \tilde U|^3_\varepsilon}\delta_{y}^{up} U,
\end{aligned}\right.
\end{equation}
where again $\tilde U=U_n$ and the spatial quantities are discretised as above and the discrete operators $\delta_{x}^{up}$ and $\delta_{y}^{up}$ are defined below. We observe that
with this choice no mixed derivative operator acting on $U$ appears. Mixed terms are encoded and considered in the previous time step. On the other hand, we get first derivative operators and not just second-order ones as in \eqref{timestep1}. Writing again the system \eqref{timestep2} as in \eqref{matrform}, we now split $F$ and $G$ in the following way:
\begin{align}
& F(U,V)=F_1(U,V)+F_2(U,V),  \label{splittinglin2}\\
& G(U,V)=G_1(U,V)+G_2(U,V), \notag
\end{align}
where
\begin{align}
& \left(\begin{array}{c} F_1(U,V) \\ G_1(U,V) \end{array}\right) = \left(\begin{array}{cc} 0 & \delta_{xx} \\
\frac{1}{|D_{\nabla}^{+} \tilde U|_\varepsilon}\delta_{xx}\frac{\delta_{x}^+ \tilde U\delta_{xx}\tilde U+\delta_{y}^+\tilde U\delta_{x}^-\delta_{y}^+\tilde U}{|D_{\nabla}^+ \tilde U|^3_\varepsilon}\delta^{up}_{x}& 0 \end{array}\right)\cdot \left(\begin{array}{c} U \\ V \end{array}\right), \label{operADIpeac} \\
& \left(\begin{array}{c} F_2(U,V) \\ G_2(U,V) \end{array}\right) = \left(\begin{array}{cc} 0 & \delta_{yy} \\
 -\frac{1}{|D_{\nabla}^+\tilde U|_\varepsilon}\delta_{yy}+\frac{\delta_{x}^+\tilde U\delta_{x}^+\delta_{y}^-\tilde U+\delta_{y}^+\tilde U\delta_{yy}\tilde U}{|D_{{\nabla}^+} \tilde U|^3_\varepsilon}\delta^{up}_{y}  & 0 \end{array}\right)\cdot \left(\begin{array}{c} U \\ V \end{array}\right). \notag
\end{align}
We note that the splitting \eqref{splittinglin2} is a two-components splitting as the one provided for the \emph{Peaceman-Rachford} ADI scheme  \eqref{peacemanrach} applied to the system \eqref{timestep2} and this gives:
\begin{equation} \label{ADIschpeaceman}
\left\{\begin{aligned}
& \left(\begin{array}{c} U_{n+1/2} \\ V_{n+1/2} \end{array}\right)=\left(\begin{array}{c} U_n+\frac{\Delta t}{2}F_1(U_n,V_n)+\frac{\Delta t}{2}F_2(U_{n+1/2},V_{n+1/2}) \\ G_1(U_n,V_n)+G_2(U_{n+1/2},V_{n+1/2})\end{array} \right), \\
& \left(\begin{array}{c} U_{n+1} \\ V_{n+1} \end{array}\right)=\left(\begin{array}{c} U_{n+1/2}+\frac{\Delta t}{2} F_1(U_{n+1},V_{n+1})+\frac{\Delta t}{2}F_2(U_{n+1/2},V_{n+1/2}) \\ G_1(U_{n+1},V_{n+1})+G_2(U_{n+1/2},V_{n+1/2})\end{array} \right).
\end{aligned}\right.
\end{equation}
For the discretisation of the first derivative operators in the scheme -- present in the equation for $V$ in \eqref{timestep2} -- we use the standard numerical technique of upwinding, i.e. the sign of the coefficients in front of the first derivatives terms affects in which direction the finite differences are computed. More precisely, we use:
\begin{align} \label{upwind}
C_1(\tilde U)\delta_{x}^{up}&=\mathbbm{1}_{\hbox{sign}(C_1>0)}\ \delta^-_{x} U+\mathbbm{1}_{\hbox{sign}(C_1<0)}\ \delta^+_{x} U, \\
C_2(\tilde U)\delta_{y}^{up} &=\mathbbm{1}_{\hbox{sign}(C_2>0)}\ \delta^-_{y} U+\mathbbm{1}_{\hbox{sign}(C_2<0)}\ \delta^+_{y} U \notag
\end{align}
where
\begin{equation*}
C_1(\tilde U)=\frac{\delta_{x}^+ \tilde U\delta_{xx}\tilde U+\delta_{y}^+\tilde U\delta_{x}^-\delta_{y}^+\tilde U}{|D^+_{\nabla} \tilde U|^3_\varepsilon},\qquad
C_2(\tilde U)=\frac{\delta_{x}^+\tilde U\delta_{x}^+\delta_{y}^-\tilde U+\delta_{y}^+\tilde U\delta_{yy}\tilde U}{|D^+_{\nabla} \tilde U|^3_\varepsilon},
\end{equation*}
and $\mathbbm{1}_{S}$ is the indicator function for the set $S$.
A numerical discussion pointing out the differences of the ADI schemes resulting from the two linearisations \eqref{timestep1} and \eqref{timestep2} follows in Section \ref{sec:numres}.

\subsubsection{Discussion of stability restrictions for the Hundsdorfer scheme}

As we are going to illustrate numerically in Section~\ref{sec:numres}, a stable application of the ADI schemes to equation \eqref{tvH-1} depends on the choice of the regularising parameter $\varepsilon$. In order to use reasonably large time steps $\Delta t$, this parameter has to be taken sufficiently large to get stable results for the numerical solution of \eqref{tvH-1}. For the following stability consideration we use the terminology introduced in Definition \ref{def_stab} where we consider the solution continuous in space and discrete in time.

Fully explicit numerical schemes solving TV gradient flows turn out to show restrictive stability conditions related to the strength of the nonlinearity in the TV subgradient, cf. \cite{ChMu,BuDiWei}.
On the other hand, fully implicit schemes solving \eqref{timestep2} without any operator splitting are unconditionally stable. In particular, we have the following stability theorem:
\begin{theorem}
Let $u_0$ be a sufficiently regular initial condition and $u_n$ the solution of
\begin{equation}\label{TVH-1implicit}
u_{n+1}=u_n-\Delta t\Delta\nabla\cdot\left(\frac{\nabla u_{n+1}}{|\nabla u_n|_\varepsilon}\right).
\end{equation}
Then, the following stability estimate holds
\begin{equation} \label{stabeveryeps}
\norm{\nabla u_{n+1}}_\varepsilon\leq \norm{\nabla u_0}_\varepsilon,
\end{equation}
where $\norm{w}_\varepsilon=\left(\int_\Omega (w^2+\varepsilon)\right)^{1/2}$ .
\end{theorem}

\begin{proof}
Multiplying equation \eqref{TVH-1implicit} by $\Delta^{-1}(u_{n+1}-u_n)$ (where $\Delta^{-1}$ is the inverse of the negative laplacian with zero boundary conditions) and integrating over $\Omega$ we get:
\begin{equation*}
\langle u_{n+1}-u_n,\Delta^{-1}(u_{n+1}-u_n)\rangle=\Delta t\langle\nabla\cdot\left(\frac{\nabla u_{n+1}}{|\nabla u_n|_\varepsilon}\right),u_{n+1}-u_n\rangle
\end{equation*}
where $\langle\cdot,\cdot\rangle$ denotes the $L^2$ inner product. We can rewrite the left hand side of the equation above using the properties of $\Delta^{-1}$ and applying the divergence theorem, thus finding:
\begin{equation*}
\langle \nabla\Delta^{-1}(u_{n+1}-u_n),\nabla\Delta^{-1}(u_{n+1}-u_n)\rangle+\Delta t\langle\left(\frac{\nabla u_{n+1}}{|\nabla u_n|_\varepsilon}\right),\nabla(u_{n+1}-u_n)\rangle=0.
\end{equation*}
We can now apply the result provided in \cite{elliott1} and summing over all $t_n=n\Delta t$ up to $T=N\Delta t$, finding the following stability estimate:
\begin{equation*}
\norm{\nabla u_{n+1}}_\varepsilon\leq \Delta t\sum_{n}\norm{\partial_t u_{n+1}}^2_{-1}+\norm{\nabla u_{N+1}}_\varepsilon\leq\norm{\nabla u_0}_\varepsilon,
\end{equation*}
where $\partial_t u_{n+1}=\frac{u_{n+1}-u_n}{\Delta t}$, which gives \eqref{stabeveryeps}. In particular, estimate \eqref{stabeveryeps} does not depend on the size of $\varepsilon$.
\end{proof}

These considerations about explicit and implicit schemes solving
directly (i.e. without any splitting) our problem \eqref{tvH-1} serve
as a motivation for the following estimates. We focus on the \emph{Hundsdorfer} scheme \eqref{ADI1} applied to the TV-$H^{-1}$ equation with the choice \eqref{ADITV}, \eqref{ADITV2}.
In each iteration the numerical solution is computed from equations consisting of a combination of explicit and implicit quantities. In particular, the explicit quantities might affect the stability properties of the scheme. To
motivate this, we focus in the following just on the first three
stages of the scheme \eqref{ADI1} applied to the TV-$H^{-1}$ equation with the choice \eqref{ADITV}, \eqref{ADITV2} and $\theta=1/2$. Considering  the first three stages of \eqref{ADI1} only can be justified by the fact that the subsequent three stages of the scheme have a similar structure and are not expected to change the stability properties drastically.

Combining the three steps of the scheme \eqref{ADI1} with \eqref{ADITV}, \eqref{ADITV2} we find the following expression:
\begin{align}
& \frac{u_{n+1}-u_n}{\Delta t} +\frac{1}{2}\partial_{xx}(C_1(u_n)\partial_{xx} u_{n+1})+\frac{1}{2}\partial_{yy}(C_2(u_n)\partial_{yy} u_{n+1}) \notag \\
& +\frac{\Delta
  t}{4}\partial_{xx}(C_1(u_n)\partial_{xx}(\partial_{yy}(C_2(u_n)\partial_{yy}
u_{n+1}))) \label{ADIschcompact}\\
=&-\Delta\nabla\cdot\left(\frac{\nabla u_n}{|\nabla u_n|_\varepsilon}\right)
 +\frac{1}{2}\partial_{xx}(C_1(u_n)\partial_{xx}
 u_{n})+\frac{1}{2}\partial_{yy}(C_2(u_n)\partial_{yy} u_{n})\notag \\
&+\frac{\Delta t}{4}\partial_{xx}(C_1(u_n)\partial_{xx}(\partial_{yy}(C_2(u_n)\partial_{yy} u_{n})))  \notag
\end{align}
where $\partial_x$ denotes the continuous derivative with respect to $x$ and the positive quantities $C_1(u_n)$ and $C_2(u_n)$ come from
the linearisation described in Section~\ref{sec:lin1}. They read:
\begin{equation*}
C_1(u_n)=\frac{\varepsilon+\partial_y(u_n)^2}{|\nabla u_n|^3_\varepsilon},\quad C_2(u_n)=\frac{\varepsilon+\partial_x(u_n)^2}{|\nabla u_n|^3_\varepsilon}.
\end{equation*}
We observe that a mixed, eighth-order operator appears, both on the
left and on the right hand side of \eqref{ADIschcompact}. In the following stability discussion we neglect these high-order terms which only represent second-order in time contributions.

 Now, we multiply equation \eqref{ADIschcompact} by $u_{n+1}$ and integrate over the domain $\Omega$. By applying integration by parts twice with respect to the $x$ and the $y$ variables to the second and the third terms of the left hand side of the equation, respectively, we get:
\begin{align} \label{est1}
&\frac{1}{2}\left(\int_\Omega \partial_{xx}(C_1(u_n)\partial_{xx} u_{n+1})u_{n+1}+\int_\Omega \partial_{yy}(C_2(u_n)\partial_{yy} u_{n+1})u_{n+1}\right) \notag\\
&=\frac{1}{2}\left(\int_\Omega C_1(u_n)(\partial_{xx} u_{n+1})^2+\int_\Omega C_2(u_n)(\partial_{yy} u_{n+1})^2\right) \\
& \geq \frac{1}{2}K(\varepsilon)(\norm{\partial_{xx}(u_{n+1})}^2+\norm{\partial_{yy}(u_{n+1})}^2). \notag
\end{align}
Here $K(\varepsilon)$ is a suitable constant that depends on the regularising parameter $\varepsilon$ only.
A similar strategy is applied to the analogous terms on the right hand side, where we have also used Young's inequality with weights $\delta_1$ and $\delta_2$. We obtain:
\begin{align}
&\frac{1}{2}\left(\int_\Omega C_1(u_n)\partial_{xx} u_n\partial_{xx}u_{n+1}+\int_\Omega C_2(u_n)\partial_{yy} u_n\partial_{yy}(u_{n+1})\right) \notag \\
&\leq \frac{1}{4\delta_1}\norm{C_1(u_n)\partial_{xx} u_{n+1}}^2+\frac{\delta_1}{4}\norm{\partial_{xx}(u_{n+1})}^2+\frac{1}{4\delta_2}\norm{C_2(u_n)\partial_{yy} u_{n+1}}^2+\frac{\delta_2}{4}\norm{\partial_{yy}(u_{n+1})}^2.\notag
\end{align}
The first term on the left hand side of \eqref{ADIschcompact} can be
dealt with by using Young's inequality again. It remains to consider the first, nonlinear, term on the right
hand side of \eqref{ADIschcompact}. By applying the divergence theorem, integration by parts and Young's inequality with weight $\delta_3$, we get for this term the following estimate:
\begin{align}  \label{estnonlin}
&-\int_{\Omega}\Delta\nabla \cdot\left(\frac{\nabla u_{n}}{|\nabla u_n|_\varepsilon}\right)u_{n+1}=\int_\Omega \nabla\nabla \cdot\left(\frac{\nabla u_{n}}{|\nabla u_n|_\varepsilon}\right)\nabla u_{n+1}\\
&=-\int_\Omega \nabla \cdot\left(\frac{\nabla u_{n}}{|\nabla u_n|_\varepsilon}\right)\Delta u_{n+1}\leq \frac{1}{2\delta_3}\norm{\nabla \cdot\left(\frac{\nabla u_{n}}{|\nabla u_n|_\varepsilon}\right)}^2+\frac{\delta_3}{2}\norm{\Delta u_{n+1}}^2. \notag
\end{align}
The second term on the right hand side of the inequality above can be merged with the corresponding ones in \eqref{est1}, choosing $\delta_1$ small enough. Denoting by $C_3(u_n)=\frac{\partial_x u_n\partial_y u_n}{|\nabla u_n|^3_\varepsilon}$, for the curvature term in \eqref{estnonlin} we observe that:
\begin{align*}
&\norm{\nabla \cdot\left(\frac{\nabla u_{n}}{|\nabla u_n|_\varepsilon}\right)}^2=\int_\Omega(C_1(u_n)\partial_{xx}(u_n)+C_2(u_n)\partial_{yy}(u_n)+C_3(u_n)\delta_{xy}(u_n))^2 \\
& \leq 2\left(2\left(\int_\Omega(C_1(u_n)\partial_xx(u_n))^2+\int_\Omega(C_2(u_n)\partial_{yy}(u_n))^2\right)+\int_\Omega(C_3(u_n)\delta_{xy}(u_n))^2\right) \\
& \leq 4\left(\frac{1}{\sqrt{\varepsilon}}+\frac{1}{|\partial_y(u_n)|}\right)^2\norm{\partial_{xx}(u_n)}^2+
4\left(\frac{1}{\sqrt{\varepsilon}}+\frac{1}{|\partial_x(u_n)|}\right)^2\norm{\partial_{yy}(u_n)}^2+\frac{1}{\sqrt{\varepsilon}}\norm{\delta_{xy}(u_n)}^2
\end{align*}
where we used Cauchy's inequality and upper bounds on $C_1, C_2$ and $C_3$.
Defining $K_1(\varepsilon)$,  $K_2(\varepsilon)$ and $K_3(\varepsilon)$ as
\begin{equation*}
K_1(\varepsilon):=4\left(\frac{1}{\sqrt{\varepsilon}}+\frac{1}{|\partial_y(u_n)|}\right)^2, \quad
K_2(\varepsilon):=4\left(\frac{1}{\sqrt{\varepsilon}}+\frac{1}{|\partial_x(u_n)|}\right)^2, \quad
K_3(\varepsilon):=\frac{1}{\sqrt{\varepsilon}}
\end{equation*}
and using an estimate proved in \cite{chen} for the mixed derivative term, we get the following bound
\begin{equation*}
\norm{\nabla \cdot\left(\frac{\nabla u_{n}}{|\nabla u_n|_\varepsilon}\right)}^2\leq K_1(\varepsilon)\norm{\partial_{xx}(u_n)}^2+ K_2(\varepsilon)\norm{\partial_{yy}(u_n)}^2+K_3(\varepsilon)\sup_{z\in\{x,y\}}\norm{\delta_{zz}(u_n)}^2.
\end{equation*}
Collecting the previous estimates, choosing $\delta_1, \delta_2$ and $\delta_3$ small enough and applying once more upper bounds on $C_1$ and $C_2$ we get the following stability estimate
\begin{align} \label{stabest}
& \frac{1}{2\Delta t}\norm{u_{n+1}}^2+\norm{\partial_{xx}(u_{n+1})}^2+\norm{\partial_{yy}(u_{n+1})}^2 \\
& \leq \frac{1}{2\Delta t}\norm{u_n}^2+\tilde{K_1}(\varepsilon)\norm{\partial_{xx}(u_n)}^2+\tilde{K_2}(\varepsilon)\norm{\partial_{yy}(u_n)}^2+
\tilde{K_3}(\varepsilon)\sup_{z\in\{x,y\}}\norm{\delta_{zz}(u_n)}^2 \notag
\end{align}
for scaled constants $\tilde{K_1}$, $\tilde{K_2}$ and
$\tilde{K_3}$ which tend to infinity as the regularising parameter $\varepsilon\searrow 0$. For this limit the estimate then blows up, thus indicating possible unstable behaviour when choosing $\varepsilon$ small. This will be discussed in more detail in Section \ref{sec:numres}.

\subsubsection{A stable AMOS ADI scheme} \label{sec:AMOS}

In order to counteract the dependence of the stability properties of the ADI schemes \eqref{ADI1} and
\eqref{ADIschpeaceman} on the size of $\varepsilon$, we now consider as an alternative the AMOS operator splitting scheme \eqref{amos1}-\eqref{amos2} for solving \eqref{systemlin1}. Due to the fully implicit character of the scheme, the hope is that stability properties improve.

To see this, both ADI methods \eqref{ADI1} and
\eqref{ADIschpeaceman} can be represented in a vectorial, multiplicative
form similar to \eqref{multsplit}. For \eqref{ADI1}, the operator $F_0$ appearing in
\eqref{multsplit} is taken explicitly in time. Thus, when writing the correspondent numerical
scheme in a multiplicative form, we have additional explicit terms on the right hand side.
Writing \eqref{ADIschpeaceman} in the form \eqref{multsplit} we again obtain additional explicit terms appearing due to the forward Euler steps in \eqref{ADIschpeaceman}. This, together with the stability estimates from the previous section, indicates possible stability restrictions for these schemes that will be confirmed by the numerical results in the following Section \ref{sec:numres}.

In the following we present an AMOS scheme for solving a slightly modified version of \eqref{tvH-1}. In system form, this new equation reads:
\begin{equation} \label{systemlin2}
\left\{\begin{aligned}
& u_t= v^1_{xx}+v^2_{yy}  \\
& v^1=-\partial_x\left(\frac{u_x}{|\nabla u|_\varepsilon}\right),\quad v^2=-\partial_y\left(\frac{u_y}{|\nabla u|_\varepsilon}\right).
\end{aligned}\right.
\end{equation}
This equation is more anisotropic than the original \eqref{tvH-1} in the sense that the nonlinear diffusion in $x$ and $y$ directions are considered separately. Only the diffusion weighting involves the whole image gradient taking both $x$ and $y$ variations into account. In this way, linearising $v_1$ and $v_2$ by considering the diffusion weighting $1/|\nabla \tilde u|_\varepsilon$ for a given $\tilde u$, results in an equation with only pure $x$ and $y$ derivatives. Considering such an equation reduces the explicit   components appearing in the scheme, allowing a fully implicit treatment of the operators, though still exploiting the advantage of directional splitting by solving along the two directions $x$ and $y$ separately.
Applying the AMOS scheme \eqref{amos1}-\eqref{amos2} to the linearisation of \eqref{systemlin2}, we obtain:
\begin{equation} \label{ADIAMOS}
\left\{\begin{aligned}
& \left(\begin{array}{c} U_{*} \\ V^2_{*} \end{array}\right)=\left(\begin{array}{c} U_n+\Delta tF_2(U_{*},V^2_{*}) \\ G_2(U_{*},V^2_{*})\end{array} \right), \\
& \left(\begin{array}{c} \tilde{U}_{n+1} \\ \tilde{V}^1_{n+1} \end{array}\right)=\left(\begin{array}{c} U_{*}+\Delta t F_1(\tilde{U}_{n+1},\tilde{V}^1_{n+1})\\ G_1(\tilde{U}_{n+1},\tilde{V}^1_{n+1})\end{array}\right) \\
& \left(\begin{array}{c} U_{\star} \\ V^1_{\star} \end{array}\right)=\left(\begin{array}{c} U_n+\Delta tF_1(U_{\star},V^1_{\star}) \\ G_1(U_{\star},V^1_{\star})\end{array} \right), \\
& \left(\begin{array}{c} \bar{U}_{n+1} \\ \bar{V}^2_{n+1} \end{array}\right)=\left(\begin{array}{c} U_{\star}+\Delta t F_2(\bar{U}_{n+1},\bar{V}^2_{n+1})\\ G_2(\bar{U}_{n+1},\bar{V}^2_{n+1})\end{array}\right), \\
& \left(\begin{array}{c} U_{n+1} \\ V^1_{n+1} \\ V^2_{n+1} \end{array}\right)= \left(\begin{array}{c} \frac{\tilde{U}_{n+1}+\bar{U}_{n+1}}{2} \\ \frac{\tilde{V}^1_{n+1}+V^1_{\star}}{2} \\ \frac{\bar{V}^2_{n+1}+V^2_{*}}{2} \end{array}\right)
\end{aligned}\right.
\end{equation}
where the operators are defined exactly as in \eqref{operADIpeac} and upwinding is used for the first derivatives as in \eqref{upwind}. We recall that the alternating application of the scheme first in the $y-x$ direction and subsequently in the $x-y$ direction allows to achieve order two of accuracy, as explained in \cite{barash}. As we will see in Section \ref{sec:numres}, the scheme \eqref{ADIAMOS} has better stability properties where the choice of the time step does not seem to depend on the size of $\varepsilon$, thus suggesting -- at least empirically -- unconditional stability.

\subsection{Primal-dual formulation of TV-$H^{-1}$ equation with penalty term} \label{sec:primdual}

An interesting alternative to the linearisations of the primal
formulation of equation \eqref{tvH-1} is motivated by earlier work on
numerically characterising elements in the subdifferential of the
total variation seminorm by primal-dual iterations, see
\cite{benning,muller,BCDS} for instance. In what follows we briefly outline such a strategy combined with ADI splitting for solving equation \eqref{tvH-1}, that is
\begin{equation} \label{pdesubdiff}
u_t=\Delta q\quad\text{in }\Omega\times (0,\infty),\quad q\in\partial |Du|(\Omega).
\end{equation}
Here, $q\in\partial |Du|(\Omega)$ by definition of the subdifferential means
\begin{equation} \label{defsubdiffTV}
q\in \partial |Du|(\Omega)\iff |Du|(\Omega)-\int_\Omega qu\,dx\leq|Dv|(\Omega)-\int_\Omega qv\,dx,\quad\forall v\in L^2(\Omega).
\end{equation}
Equivalently, if $u\in BV(\Omega)\subset L^2(\Omega)$ solves the variational problem
\begin{equation} \label{minprob}
\min_{v\in BV(\Omega)}\left\{|Dv|(\Omega)-\int_\Omega qv\,dx\right\},
\end{equation}
then, by definition of being a minimum, $u$ fulfills \eqref{defsubdiffTV}, that is $q\in\partial|Du|(\Omega)$. Inserting the definition of the total variation seminorm \eqref{TVfunct} into \eqref{minprob} we receive
\begin{equation} \label{minprob1}
\min_{v\in BV(\Omega)}\left\{\sup_{\textbf{p}\in C_0^{\infty}(\Omega;\R^2),\ \norm{\textbf{p}}_\infty\leq 1}\int_\Omega  \nabla\cdot\textbf{p} v\, dx-\int_\Omega qv\,dx\right\}
\end{equation}
which is typically known as the \emph{primal-dual} formulation of the problem \eqref{pdesubdiff}. The constraint on \textbf{p} appearing in \eqref{minprob1} can be relaxed, for instance, by a penalty method. That is, we remove the constraint from the minimisation in \eqref{minprob1} and instead add a term that penalises the functional if $\norm{\textbf{p}}_\infty>1$. A typical example for such a penalty term $F$ is
\begin{equation*}
F(s)=\frac{1}{2}\norm{\max\{s,0\}}^2_2.
\end{equation*}
With these considerations we reformulate \eqref{minprob1} as
\begin{equation} \label{minprob2}
\min_{v\in BV(\Omega)}\sup_{\textbf{p}\in C_0^\infty(\Omega;\R^2)}\left\{\int_\Omega\nabla\cdot\textbf{p} v\, dx-\frac{1}{\varepsilon}F(|\textbf{p}|-1)-\int_\Omega q v\, dx\right\}
\end{equation}
where the parameter $1\gg \varepsilon>0$ is the weight of our penalisation. We can then find the optimality conditions for solutions $\textbf{p}$ and $u$ of \eqref{minprob2} which, merged with equation \eqref{pdesubdiff}, allow us to consider the following, approximate formulation of \eqref{tvH-1}
\begin{equation}  \label{primdualsyst}
\left\{\begin{aligned}
 u_t & =\Delta q, \\
 q & =\nabla\cdot\textbf{p}, \\
 0 & = -\nabla u-\frac{1}{\varepsilon}H(\textbf{p}).
 \end{aligned}\right.
\end{equation}
In the system above we have indicated by $H$ the derivative of the penalty term $F(|\textbf{p}|-1)$, i.e.
$$
H(\textbf p) = \mathbbm{1}_{\{|\textbf p|\geq 1\}} \mathrm{sgn}(\textbf p) (|\textbf p|-1),
$$
which we linearise via its first-order Taylor approximation, that is
\begin{equation*}
H(\textbf p)\approx H(\tilde{\textbf p}) + H'(\tilde{\textbf p})(\textbf p-\tilde{\textbf p}).
\end{equation*}
Here, $H'$ denotes the Jacobian of $H$. In order to guarantee the invertibility of the now linear operator that defines the system \eqref{primdualsyst} we add an additional damping term in $\mathbf p$ as suggested, for instance, in \cite{muller}. Collecting everything, we propose the following numerical scheme for solving \eqref{primdualsyst},
\begin{equation}  \label{primdualsystlin}
\left\{\begin{aligned}
  \frac{U^{(k)}_{n+1}-U_{n}}{\Delta t}&=D_{\Delta} Q^{(k)}_{n+1}, \\
  Q^{(k)}_{n+1}&=D^{-}_{\hbox{div}}\textbf{P}^{(k)}_{n+1},\\
  0 =& -D^+_{\nabla} U^{(k)}_{n+1}-\frac{1}{\varepsilon}H(\textbf{P}^{(k-1)}_{n+1})\\
  &-\frac{1}{\varepsilon}H'(\textbf{P}^{(k-1)}_{n+1})(\textbf{P}^{(k)}_{n+1}-\textbf{P}^{(k-1)}_{n+1})-\tau^k(\textbf{P}^{(k)}_{n+1}-\textbf{P}^{(k-1)}_{n+1}),
 \end{aligned}\right.
\end{equation}
which consists of two nested iterations, the inner damped Newton iteration
with index $k$ and an outer time stepping with index $n$ for the evolution in time of $U$. Here, $\tau^k$
is a sequence of parameters controlling the damping of the Newton iteration in every time step. The authors in \cite{muller,BCDS} suggest to start with a large value $\tau^0$ and then decrease it in the inner iterations to ensure efficient convergence.

System \eqref{primdualsystlin} could now be discretised in space and either solved directly using Newton's iterations (compare \cite{BCDS} where the authors used this strategy to solve equation \eqref{TVWassflow}) or fitted into a numerical ADI scheme. A detailed presentation and
analysis of the resulting scheme is beyond the scope of the present paper and a matter of future research.

\section{Numerical results}  \label{sec:numres}
\setcounter{equation}{0}

In this section we discuss the numerical performance of the ADI methods
proposed in this paper. To do so, we present numerical experiments for a Gaussian, an oscillatory function made up of sine and cosine functions and grayscale images as initial conditions. The paper is furnished with numerical examples for image inpainting.

We start presenting the numerical results
obtained applying the \emph{Hundsdorfer} scheme \eqref{ADIgen} to compute
the numerical solution of the biharmonic equation
\eqref{biharm} and to the equivalent system \eqref{biharm1}. Next, we report on numerical experiments for the TV-$H^{-1}$ equation \eqref{tvH-1} using the
\emph{Hundsdorfer} scheme \eqref{ADI1} with \eqref{ADITV}-\eqref{ADITV2}. The choice of the time step size $\Delta t$ of this scheme is constrained by very strong stability restrictions related to the size of the regularising parameter $\varepsilon$. A possible dependence of this type for stability was already discussed in Section~\ref{sec:ADIappl}. Our numerical tests for the application of the \emph{Peaceman-Rachford} scheme \eqref{ADIschpeaceman} to the TV-$H^{-1}$ equation show the same stability behaviour and are therefore not included in the paper. Finally, we show the numerical results for the solution of the modified TV-$H^{-1}$ system \eqref{systemlin2} solved with the AMOS scheme \eqref{ADIAMOS}. This scheme shows stable behaviour independent of the size of $\Delta t$ and $\varepsilon$.

\subsection{The biharmonic equation: numerical results} \label{subsec:biharm}

The linear system that arises from the application of the \emph{Hundsdorfer} scheme \eqref{ADI1} with \eqref{ADIbiharm} to the biharmonic equation \eqref{biharm1} is solved by the Schur complement method.

We consider a grid of $100\times100$ grid points for the discretisation of the spatial domain $\Omega$ being the
unit square. We analyse the example of
the evolution of the biharmonic equation having as initial condition
$U_0$ the Gaussian density
$U^0_{ij}=\exp{(-((x_i-1/2)^2}$ $+(y_j+1/2)^2)/\gamma^2)$ where the
variance $\gamma^2$ is equal to $100$. Figure \ref{ADI2} shows two iterates of the scheme with $\theta=\sigma=1/2$ for a time-step size $\Delta t=C(\Delta x)^2$, where $C$ now and for the rest of the discussion is equal to $0.1$. Even for such a big choice of $\Delta t$ the result is stable and the convergence of the solution to the steady state is quick. Considering another initial condition $U_0$ and taking, for instance, some very oscillatory function does not effect the performance of the method. This confirms the unconditional stability of this scheme when applied to a linear fourth-order equation such as the biharmonic equation \eqref{biharm}, compare \cite{hout1}. In the following section we will see that the unconditional stability of the \emph{Hundsdorfer} scheme breaks down when applied to the nonlinear fourth-order diffusion equation TV-$H^{-1}$ \eqref{tvH-1}.

\begin{figure}[!h]
\begin{subfigure}{0.3\textwidth}
\includegraphics[height=3.5cm,width=4.5cm]{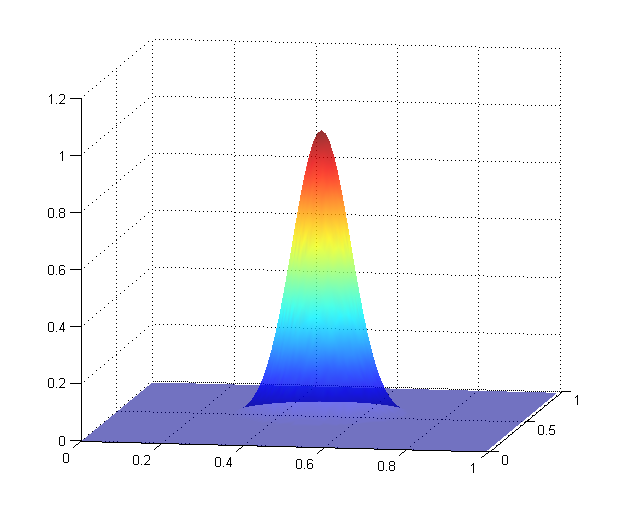}
\caption{Gaussian initial cond.}
\end{subfigure}
\hspace{0.5cm}
\begin{subfigure}{0.3\textwidth}
\includegraphics[height=3.5cm,width=4.5cm]{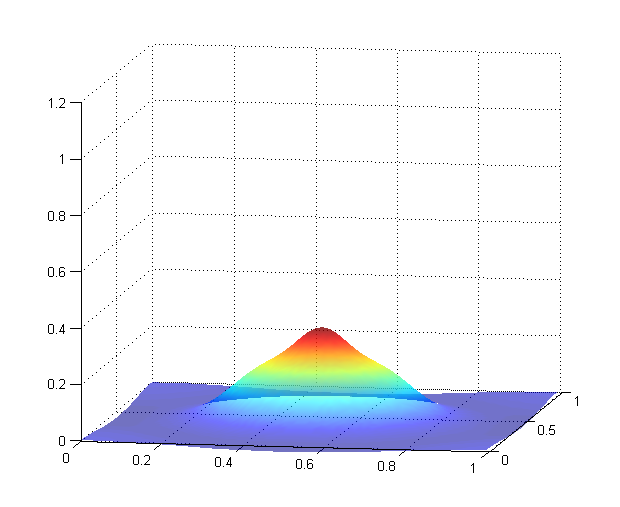}
\caption{Solution $U_4$}
\end{subfigure}
\begin{subfigure}{0.3\textwidth}
\includegraphics[height=3.5cm,width=4.5cm]{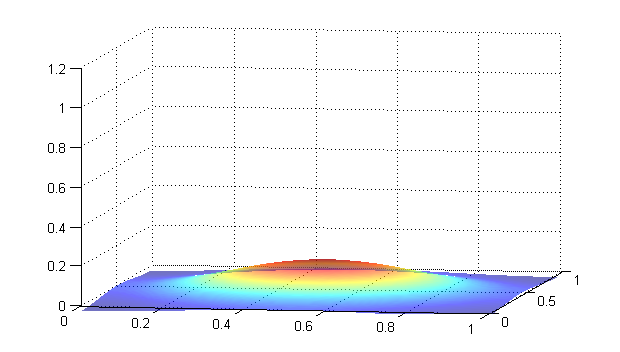}
\caption{Solution $U_{20}$}
\end{subfigure}
\caption{Evolution of the biharmonic equation \eqref{biharm1} solved with the \emph{Hundsdorfer} scheme \eqref{ADI1} with $\Delta t=C(\Delta x)^2$.}
\label{ADI2}
\end{figure}

\subsection{TV-$H^{-1}$ equation: numerical results} \label{subsec:restv}

In what follows we provide numerical discussion for applying the directional splitting schemes introduced in Section \ref{sec:ADIappl} for the numerical solution of the regularised TV-$H^{-1}$ equation \eqref{systemlin1}. In particular we consider \emph{Hundsdorfer} ADI scheme \eqref{ADI1} with \eqref{ADITV}, \eqref{ADITV2} and the AMOS scheme
\eqref{ADIAMOS}. 
Once again we use the Schur complement technique to solve the linear systems that arise in the numerical solution of these schemes.  The
expected edge-preserving behaviour of the TV-$H^{-1}$ equation \eqref{tvH-1} which is due to the subgradient of the TV functional
is closely related to the size of the regularising parameter
$\varepsilon$ used in \eqref{systemlin1}. This parameter, in fact, becomes a
``measure" of how close the nonlinear diffusion is to the linear,
biharmonic one. Choosing $\varepsilon$ too big the smoothing behaviour of \eqref{systemlin1} is similar to
the one of the biharmonic equation \eqref{biharm}. In this case the stability properties of the \emph{Hundsdorfer} scheme applied to the nonlinear equation are close to the ones discussed for the biharmonic equation in the previous section. On the
other hand, small values of $\varepsilon$ keep the regularised version of the subgradient of the total variation close to its exact characterisation and hence solution show edge-preserving features. However, stability issues that disturb and worsen the
performance of the methods appear. The explicit treatment of some of the terms in the \emph{Hundsdorfer} and the \emph{Peaceman-Rachford} schemes (namely, the mixed derivative term
in \eqref{ADI1} and the half-direction forward Euler steps in
\eqref{ADIschpeaceman}) seem to influence the stability of the schemes in a negative way. In particular, the time step sizes $\Delta t$ have to be decreased significantly with small values of $\varepsilon$. Moreover, the choice of the initial condition also influences the stability properties. For instance, for smooth Gaussian initial conditions with large support the time steps can be chosen larger than for an oscillatory initial condition, see Figures \ref{figadi1}-\ref{figadi3}. These issues are resolved by the AMOS scheme \eqref{ADIAMOS} which did not show dependence of $\Delta t$ on the size of $\varepsilon$ nor on the type of initial condition in order to get stable results.

In the following we present the numerical results obtained considering
the linearisation of the system \eqref{systemlin1} given by
\eqref{timestep1} and solved by the \emph{Hundsdorfer} ADI method
\eqref{ADI1}. We consider as initial conditions both the Gaussian density
$U^0_{ij}=\exp{(-((x_i-1/2)^2}$ $+(y_j+1/2)^2)/\gamma^2)$ with
$\gamma^2=100$ and the oscillatory function $U^0_{ij}=\sin(8\pi
x_i)+\cos(8\pi y_j)$. Figure \ref{figadi1} shows iterates of the \emph{Hundsdorfer} scheme with $\theta=\sigma=1/2, \varepsilon=5$ and $\Delta t=C(\Delta
x)^3$ applied to the Gaussian datum. The value $\varepsilon=5$ is the smallest possible value that can be used in order to get stable solutions of the \emph{Hundsdorfer} scheme with $\Delta t=C(\Delta x)^3$. Figure \ref{figadi3} shows the evolution of the process for the
oscillatory datum with the same choice of $\varepsilon$ as before. In this case to
get stable results smaller time steps are needed, thus showing stability dependence
on the initial condition.

\begin{figure}[!h]
\begin{subfigure}{0.3\textwidth}
\includegraphics[height=3.5cm,width=4.5cm]{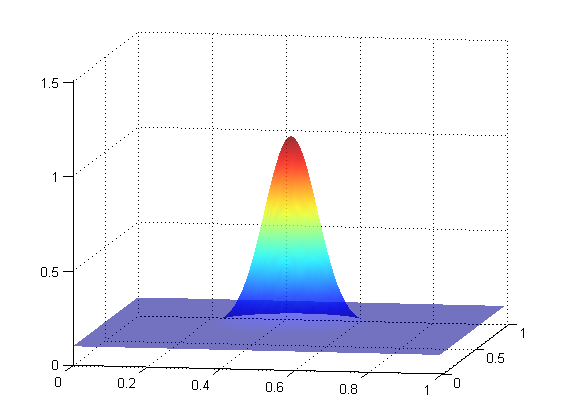}
\caption{Gaussian initial cond.}
\end{subfigure}
\hspace{0.5cm}
\begin{subfigure}{0.3\textwidth}
\includegraphics[height=3.5cm,width=4.5cm]{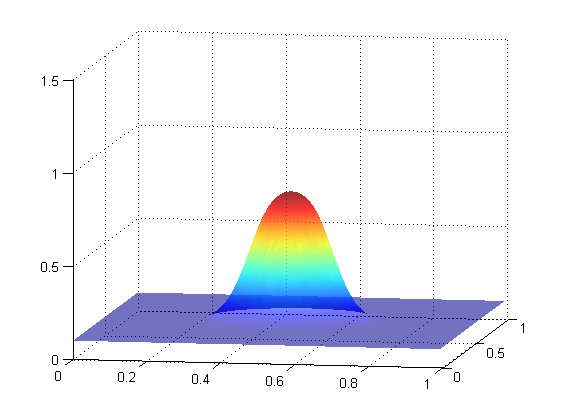}
\caption{Solution $U_{100}$}
\end{subfigure}
\begin{subfigure}{0.3\textwidth}
\includegraphics[height=3.5cm,width=4.5cm]{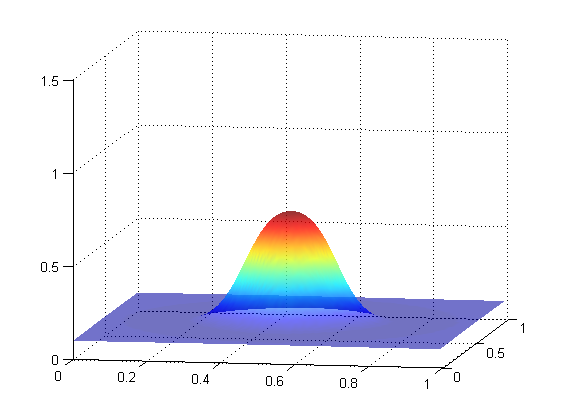}
\caption{Solution $U_{200}$}
\end{subfigure}
\caption{Evolution of the TV-$H^{-1}$ equation \eqref{tvH-1} by the \emph{Hundsdorfer} scheme \eqref{ADI1} with $\Delta t=C(\Delta x)^3$ and $\varepsilon=5$.}
\label{figadi1}
\end{figure}

\begin{figure}[!h]
\begin{subfigure}{0.3\textwidth}
\includegraphics[height=3.5cm,width=4.5cm]{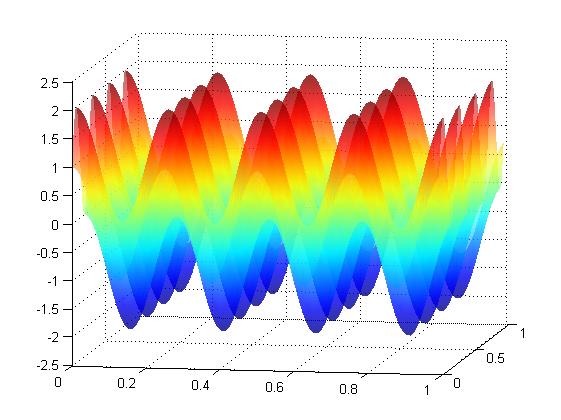}
\caption{Oscillatory int. cond.}
\end{subfigure}
\hspace{0.5cm}
\begin{subfigure}{0.3\textwidth}
\includegraphics[height=3.5cm,width=4.5cm]{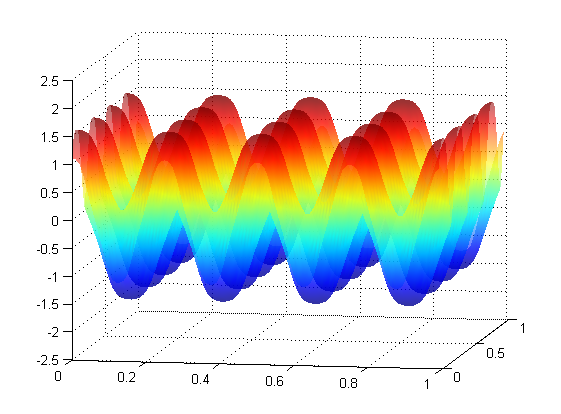}
\caption{Solution $U_{2500}$}
\end{subfigure}
\begin{subfigure}{0.3\textwidth}
\includegraphics[height=3.5cm,width=4.5cm]{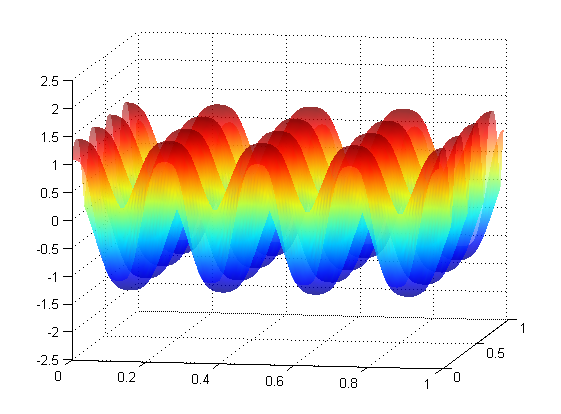}
\caption{Solution $U_{5000}$}
\end{subfigure}
\caption{Evolution of the TV-$H^{-1}$ equation \eqref{tvH-1} by the \emph{Hundsdorfer} scheme \eqref{ADI1} with decreased $\Delta t=C(\Delta x)^4$ and $\varepsilon=5$ to control stability.}
\label{figadi3}
\end{figure}

In Figures \ref{plotstability1}--\ref{plotstability2} we describe this stability dependence in more detail. For the two different choices of the initial conditions considered above, the smallest values $\varepsilon$ needed for stability as the size of $\Delta t$ increases are plotted. To compare the different setups of the \emph{Hundsdorfer} scheme with respect to $\theta$, these tests were performed for different values of the stabilising parameter $\theta=0,1/2,1/2+\sqrt{3}/6,1$ thus resulting into four graphs per test. Note that different values of $\theta$ (compare \cite{hunds,hunds1,hout1,hout2} for such choices) result into different weighting of implicit and explicit terms, that means considering a fully explicit method for $\theta=0$ and implicit contributions in the unidirectional steps of \eqref{ADI1} in the other cases (see Section \ref{subsec_HUNDS}). For each of these graphs, their epigraph corresponds to the region of \emph{stability} of the method (according to Definition \ref{def_stab}). To obtain such graphs we have considered different values of the constant $C$ and of the regularising parameter $\varepsilon$ for time step sizes of the order $(\Delta x)^k, k=2,3,4$. For the choice $\theta=0$ stable solutions are only obtained by very restrictive choices of $\varepsilon$. This situation improves for $\theta>0$. In particular, for values of $\theta$ close to $1$ the stability constraint on the size of $\varepsilon$ is reduced. However, in all cases, these plots show a clear dependence of the stability of the \emph{Hundsdorfer} scheme on the strength of the nonlinearity (encoded by the size of $\varepsilon$). This creates numerical difficulties in the attempt of increasing the time step size to get an efficient numerical scheme for solving the approximated problem \eqref{systemlin1} for sufficiently small values of $\varepsilon$.

\begin{figure}[!h]
\centering
\includegraphics[height=5cm]{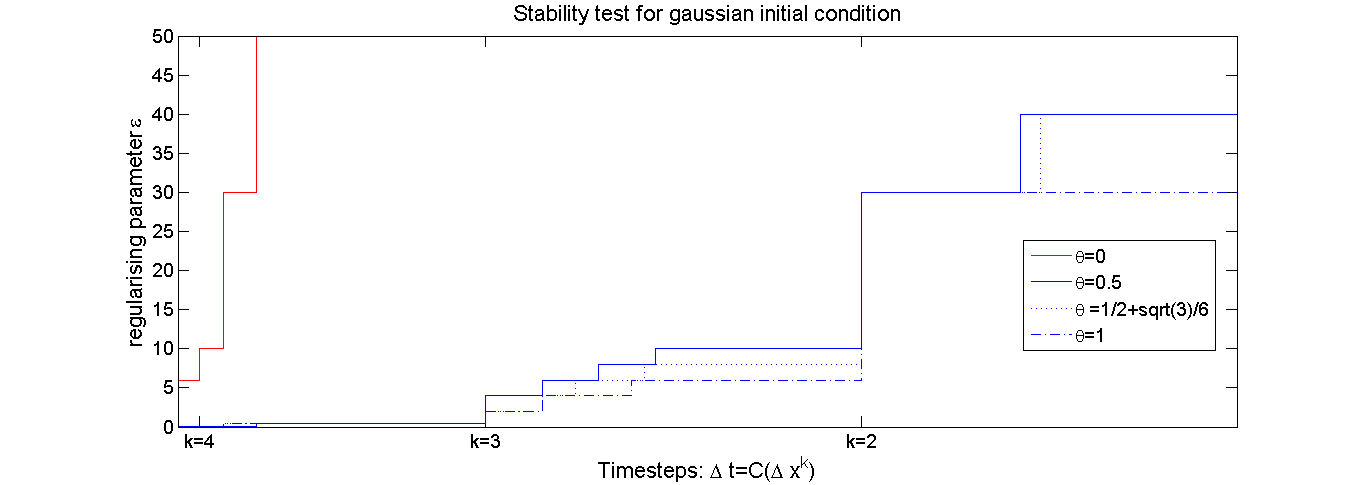}
\caption{Stability test for the numerical solution of \eqref{timestep1} solved with the \emph{Hundsdorfer} scheme \eqref{ADI1} with initial condition $U^0_{ij}=\exp{-(((x_i-1/2)^2+(y_j+1/2)^2)/\gamma^2)}$ with $\gamma^2=100$ for different choices of stabilising parameter $\theta$. For each time step size the minimum value of $\varepsilon$ for which we get stability is plotted.}
\label{plotstability1}
\end{figure}

\begin{figure}[!h]
\centering
\includegraphics[height=5cm,width=14cm]{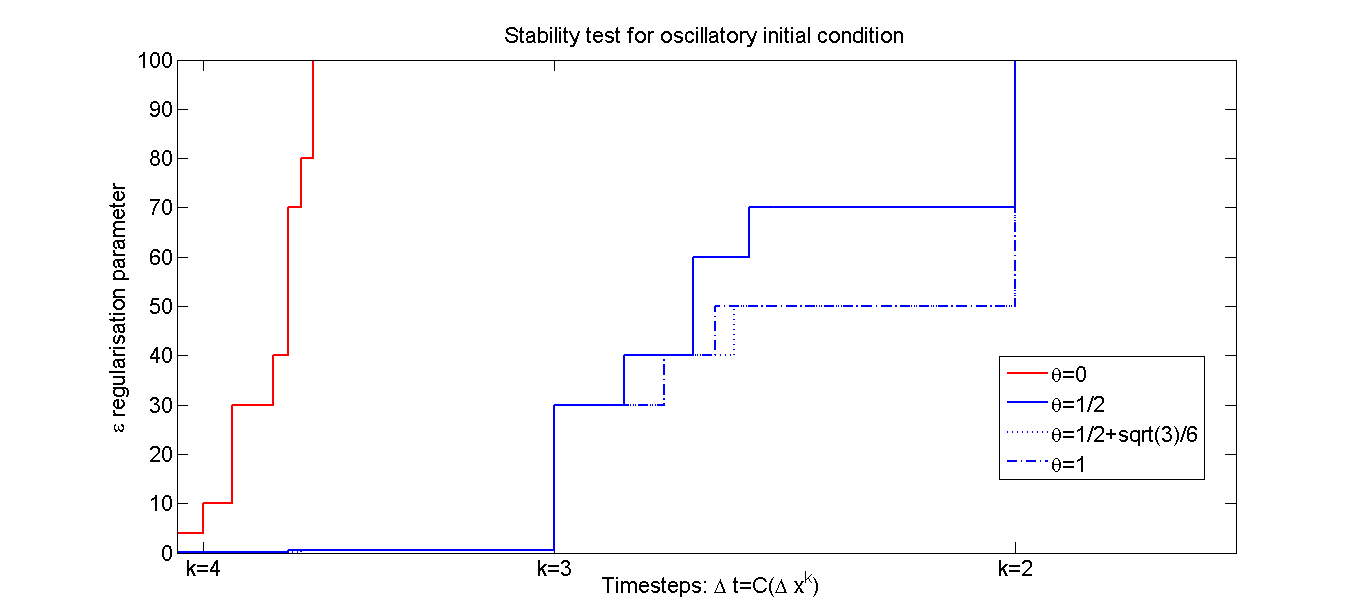}
\caption{Stability test for the numerical solution of \eqref{timestep1} solved with the \emph{Hundsdorfer} scheme \eqref{ADI1} with initial condition $U^0_{ij}=\sin(8\pi x_i)+\cos(8\pi y_j)$. Comparison with Figure \ref{plotstability1} shows dependence on the initial condition for admissible values of $\varepsilon$ providing stability.}
\label{plotstability2}
\end{figure}

We do not present here the numerics related to the application of the \emph{Peaceman-Rachford} method \eqref{ADIschpeaceman} to the TV-$H^{-1}$ equation \eqref{systemlin2} as the stability issues resemble the ones described above. In order to overcome such problems, we present in the following the results related to the application of the ADI AMOS scheme \eqref{ADIAMOS} solving the slightly modified system \eqref{systemlin2}.

Due to the implicit character of the scheme \eqref{ADIAMOS}, improved stability properties are expected while keeping the advantages of the directional splitting strategy, that is each can be solved very efficiently. In the Figures \ref{figAMOS1}-\ref{figAMOS2} we show the evolution of the TV-$H^{-1}$ equation \eqref{systemlin2} for the Gaussian initial condition as above for $\Delta t=C(\Delta x)^3$ and $\Delta t=C(\Delta x)^2$ with fixed $\varepsilon=0.001$. We observe that the time discretisation provides stable results even for large $\Delta t$. However, as clearly visible in Figure \ref{figAMOS2}, choosing $\Delta t$ too large badly affects time accuracy.

\begin{figure}[!h]
\begin{center}
\includegraphics[height=5cm]{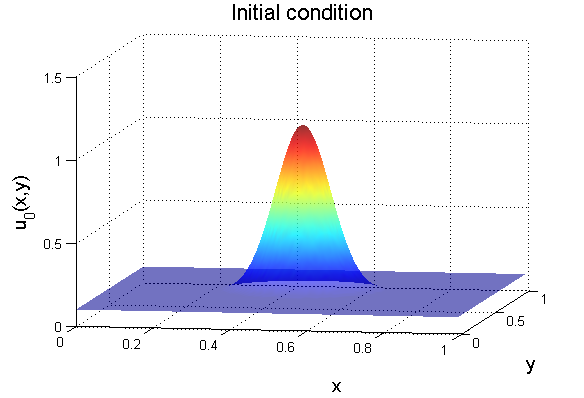}
\end{center}
\caption{Initial condition}
\end{figure}

\begin{figure}[!h]
\begin{subfigure}{0.3\textwidth}
\includegraphics[height=3.5cm,width=4.5cm]{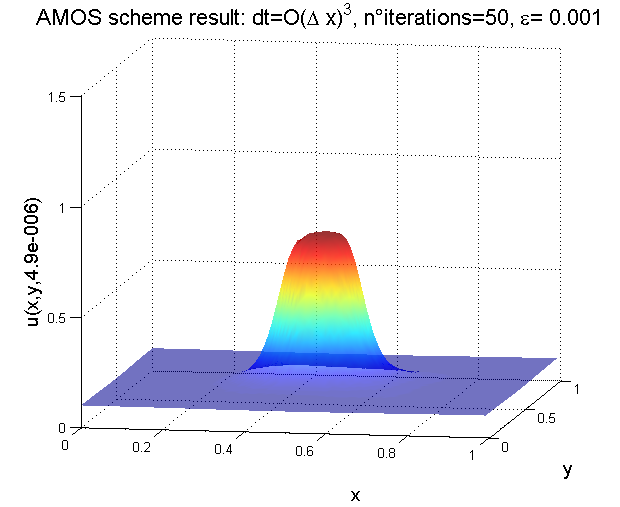}
\caption{Solution $U_{50}$}
\end{subfigure}
\hspace{0.5cm}
\begin{subfigure}{0.3\textwidth}
\includegraphics[height=3.5cm,width=4.5cm]{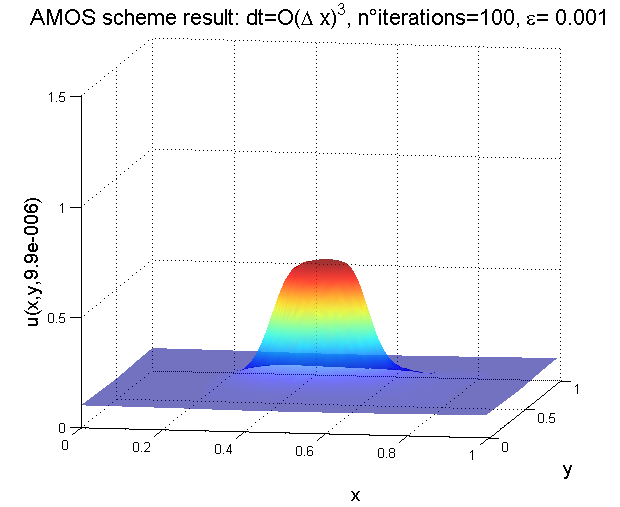}
\caption{Solution $U_{100}$}
\end{subfigure}
\begin{subfigure}{0.3\textwidth}
\includegraphics[height=3.5cm,width=4.5cm]{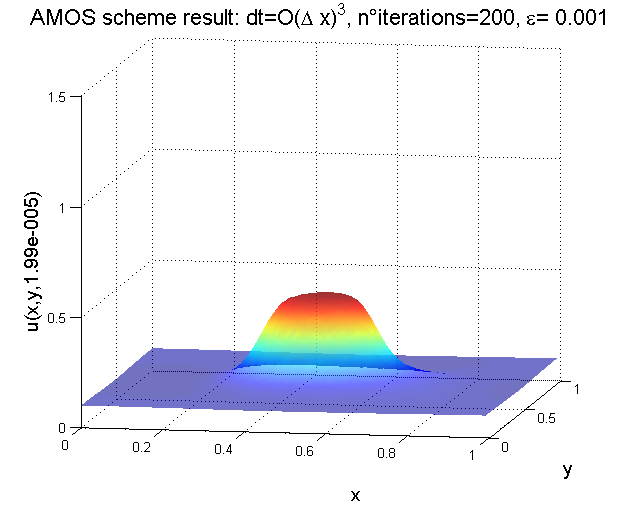}
\caption{Solution $U_{200}$}
\end{subfigure}
\caption{Evolution of the modified TV-$H^{-1}$ equaation \eqref{systemlin2} with the AMOS scheme \eqref{ADIAMOS} with $\Delta t=C(\Delta x)^3$ and $\varepsilon=0.001$.}
\label{figAMOS1}
\end{figure}

\begin{figure}[!h]
\begin{subfigure}{0.3\textwidth}
\includegraphics[height=3.5cm,width=4.5cm]{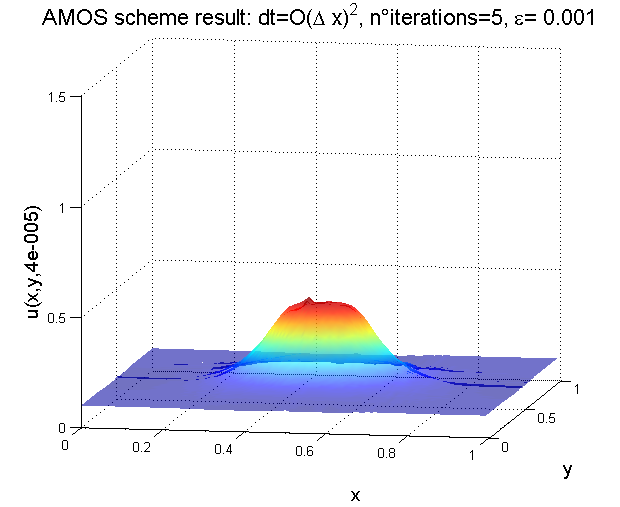}
\caption{Solution $U_{5}$}
\end{subfigure}
\hspace{0.5cm}
\begin{subfigure}{0.3\textwidth}
\includegraphics[height=3.5cm,width=4.5cm]{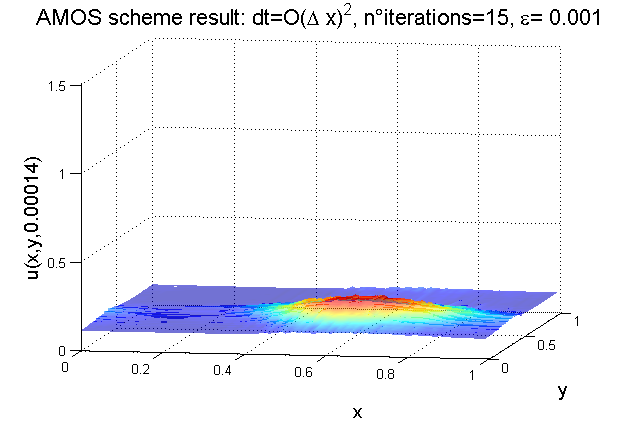}
\caption{Solution $U_{15}$}
\end{subfigure}
\begin{subfigure}{0.3\textwidth}
\includegraphics[height=3.5cm,width=4.5cm]{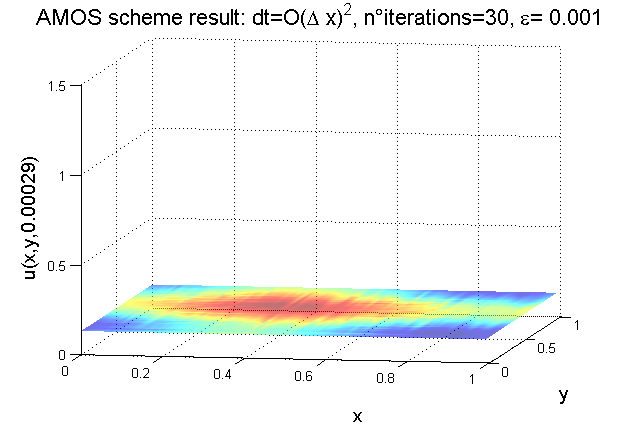}
\caption{Solution $U_{30}$}
\end{subfigure}
\caption{Evolution of the modified TV-$H^{-1}$ equaation \eqref{systemlin2} with the AMOS scheme \eqref{ADIAMOS} with $\Delta t=C(\Delta x)^2$ and $\varepsilon=0.001$.}
\label{figAMOS2}
\end{figure}

The convergence rate of the AMOS scheme \eqref{ADIAMOS} for the choice of $\Delta t=C(\Delta x)^3$ and regularising parameter $\varepsilon=0.001$ when applied to the Gaussian initial condition is presented in Figure \ref{conv:steady}. We observe an exponential-type decay for the $\ell^\infty$ norm of the difference between the iterative solution $U_n$ of \eqref{ADIAMOS} and the steady state $U_\infty$ which has been computed numerically beforehand by iterating the scheme till the quantity $\norm{U_{n+1}-U_{n}}_{\infty}/(MN)\leq10^{-10}$. For comparison, the exponential function $(7\cdot10^{-6}) e^{-0.05x}$, has been plotted. Figure \ref{energy:decay} shows the decay of total variation energy \eqref{TVfunct} towards the energy of the steady state $U_\infty$ against the number of iterations. This shows that although the modified TV-$H^{-1}$ equation \eqref{systemlin2} is not as the same as the gradient flow of the total variation in the space $H^{-1}$, the total variation energy is still decreasing in every iteration.

\begin{figure}[!h]
\centering
\includegraphics[height=6cm,width=14cm]{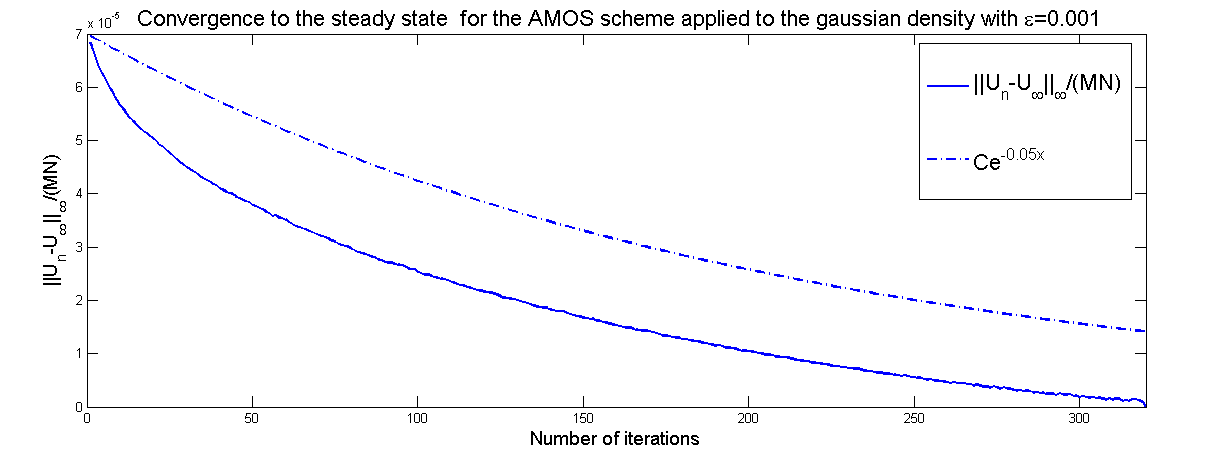}
\caption{Convergence to the steady state for the numerical solution of the modified TV-$H^{-1}$ equation \eqref{systemlin2} computed with the AMOS scheme \eqref{ADIAMOS}.}
\label{conv:steady}
\end{figure}

\begin{figure}[!h]
\centering
\includegraphics[height=6cm,width=15cm]{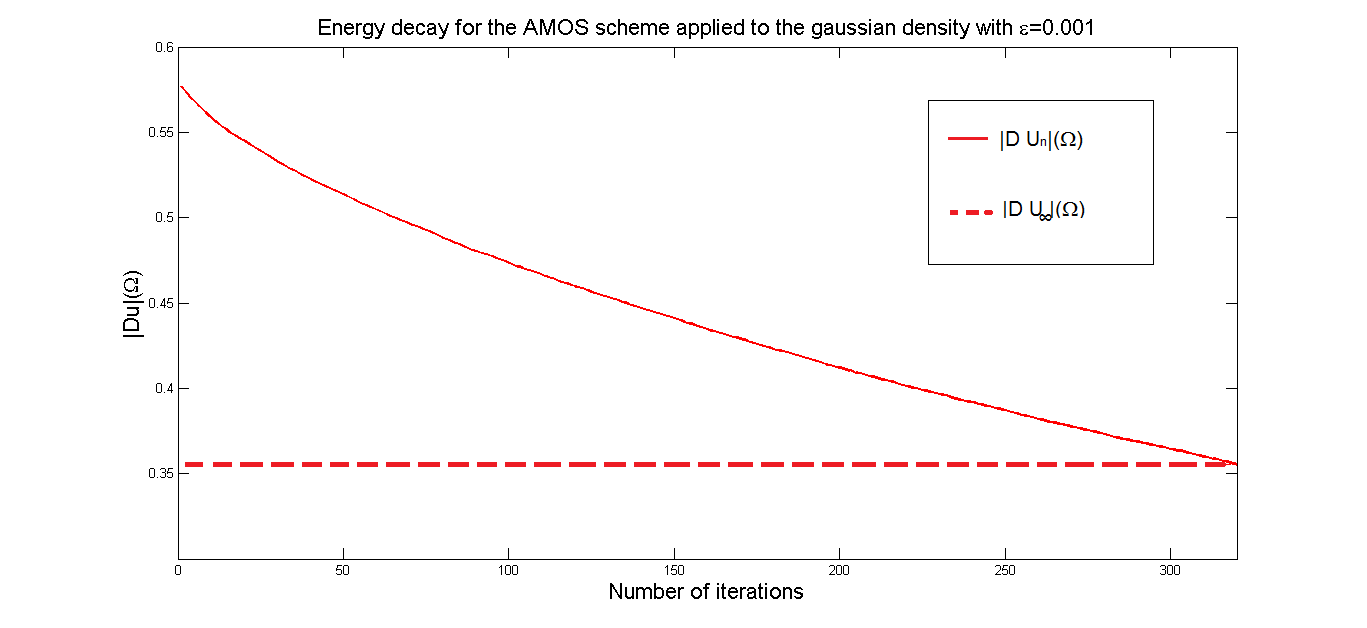}
\caption{Energy decay for the numerical solution of the modified TV-$H^{-1}$ equation \eqref{systemlin2} computed with the AMOS scheme \eqref{ADIAMOS}.}
\label{energy:decay}
\end{figure}
\newpage

Motivated by our original purposes of applying the ADI schemes in the imaging framework, we present in Figure \ref{scalespace} the scale space properties of system \eqref{systemlin2} solved with the AMOS scheme \eqref{ADIAMOS} for a  $120\times 120$ image. The time step size is $\Delta t=C(\Delta x)^3$ and the regularising parameter $\varepsilon=0.001$. Due the nonlinear nature of the equation, the diffusion is anisotropic.

\begin{figure}[!h]
\begin{subfigure}{0.3\textwidth}
\centering
\includegraphics[height=2.5cm]{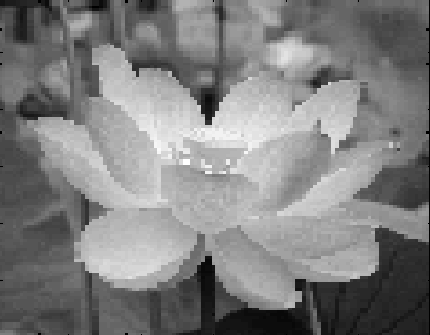}
\caption{Initial condition.}
\end{subfigure}
\quad
\begin{subfigure}{0.3\textwidth}
\centering
\includegraphics[height=2.5cm]{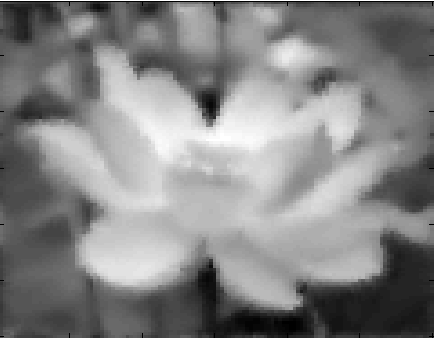}
\caption{Solution $U_{2}$.}
\end{subfigure}
\quad
\begin{subfigure}{0.3\textwidth}
\centering
\includegraphics[height=2.5cm]{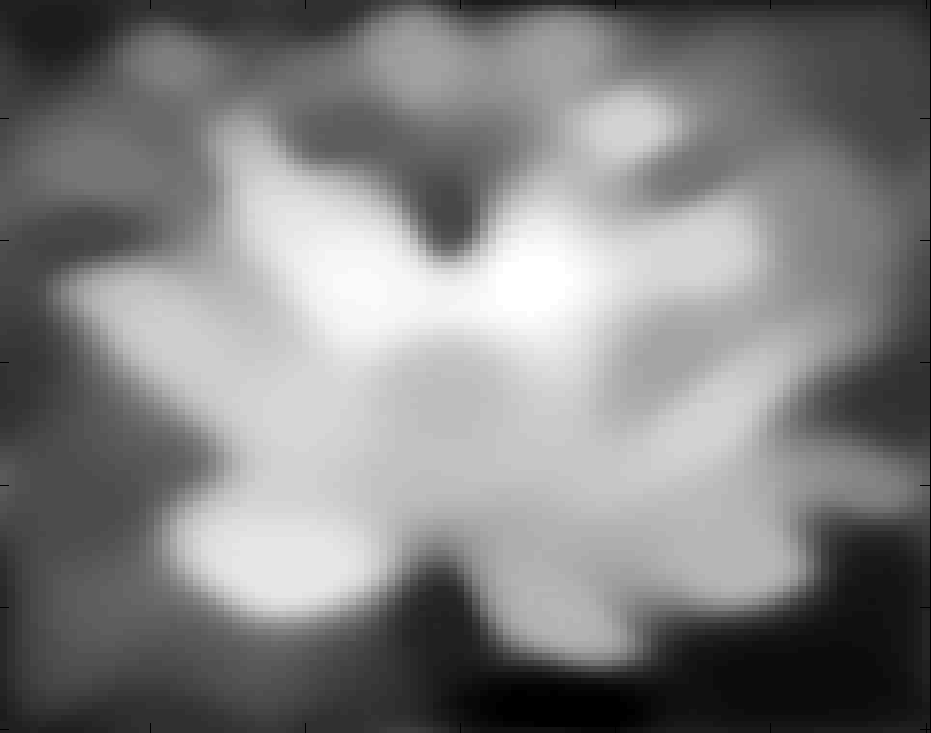}
\caption{Solution $U_{10}$.}
\end{subfigure}\\
\begin{subfigure}{0.3\textwidth}
\centering
\includegraphics[height=2.5cm]{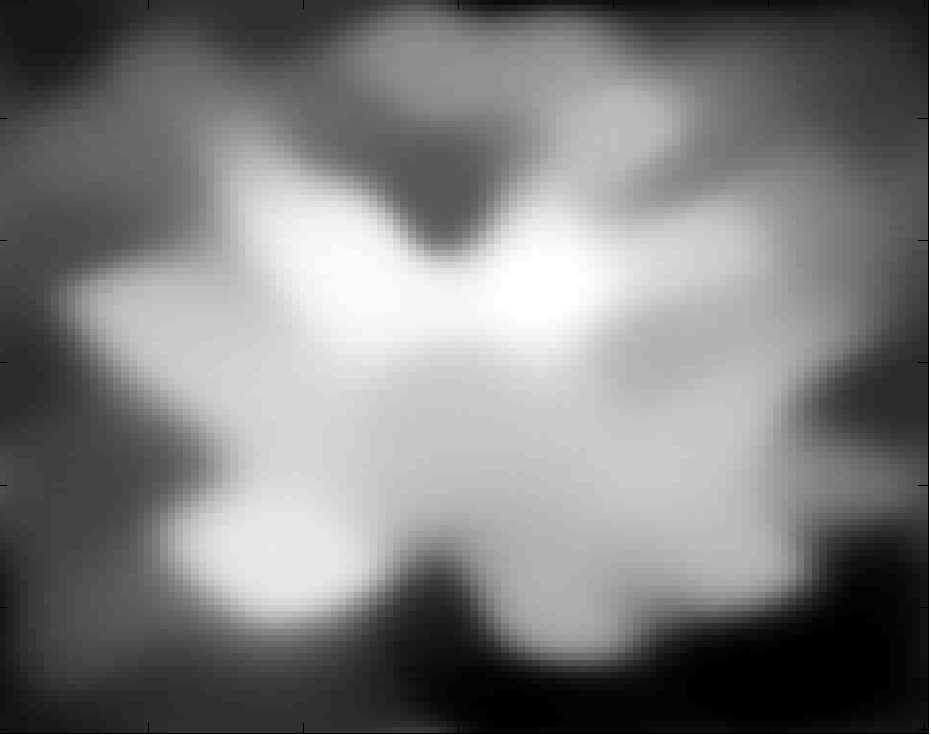}
\caption{Solution $U_{20}$.}
\end{subfigure}
\quad
\begin{subfigure}{0.3\textwidth}
\centering
\includegraphics[height=2.5cm]{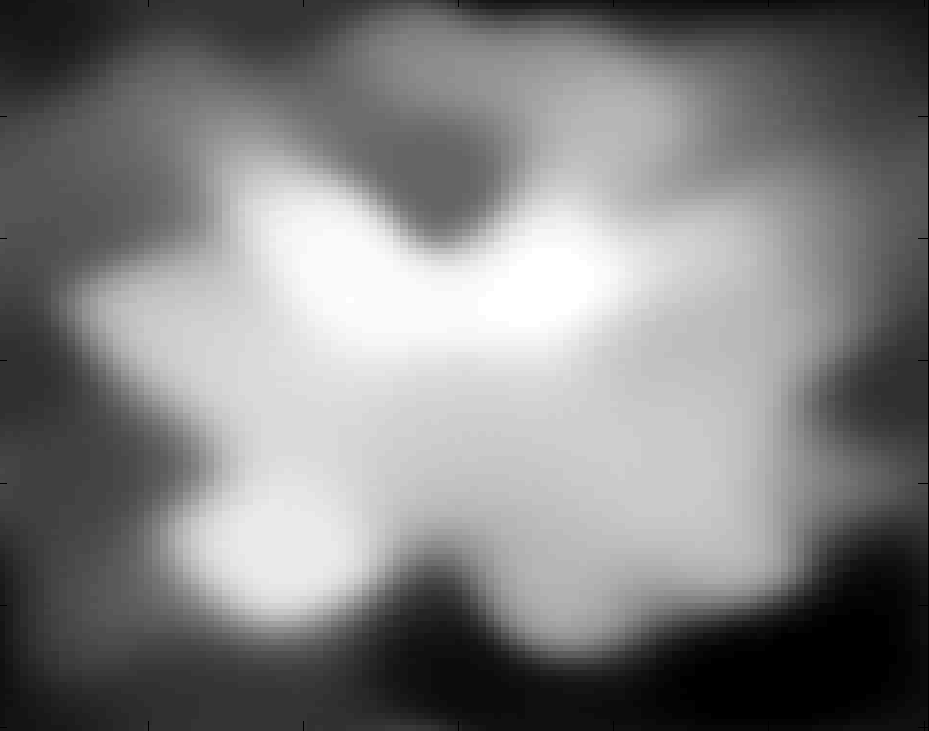}
\caption{Solution $U_{30}$.}
\end{subfigure}
\quad
\begin{subfigure}{0.3\textwidth}
\centering
\includegraphics[height=2.5cm]{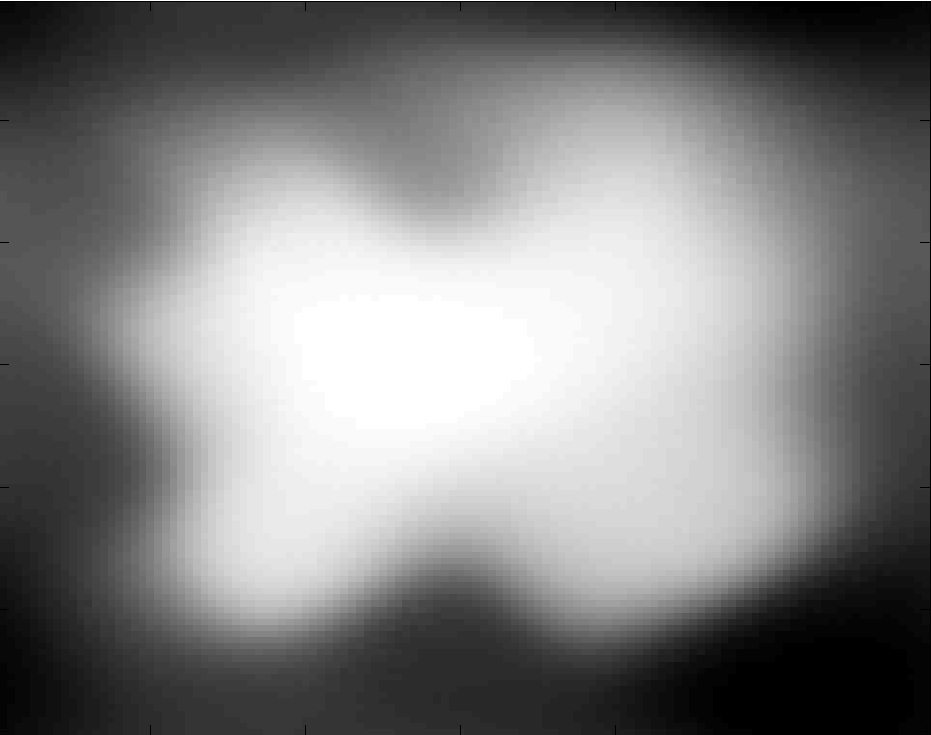}
\caption{Solution $U_{120}$.}
\end{subfigure}
\caption{Scale space properties for the iterates of the numerical solution of \eqref{systemlin2} computed with the AMOS scheme \eqref{ADIAMOS}, $\Delta t=C(\Delta x)^3,\  \varepsilon=0.001$.}
\label{scalespace}
\end{figure}


Finally, we present some numerical results obtained by using the AMOS scheme \eqref{ADIAMOS} to solve the $H^{-1}$-inpainting model \eqref{tvh-1inp} with the modified TV-$H^{-1}$ equation given by \eqref{systemlin2}. The constraint $u=f$ outside of the inpainting domain is approximately enforced by adding the fidelity term $\lambda\cdot\mathbbm{1}_{\Omega\setminus D}(f-u)$ to the equation, where $\lambda$ measures how close the reconstructed image is to the original one. In Figure \ref{figinp1} we show the result for inpainting a $150\times 150$ image of a cross: the initial condition is inpainted in $1000$ iterations using a time step size of$\Delta t=C(\Delta x)^3$ and the regularising parameter $\varepsilon$ is chosen to be $\varepsilon=0.001$, thus allowing the preservation of edges additionally to fulfilling the connectivity principle that is a consequence of the fourth-order of the method.

\begin{figure}[!h]
\begin{center}
\includegraphics[width=5cm]{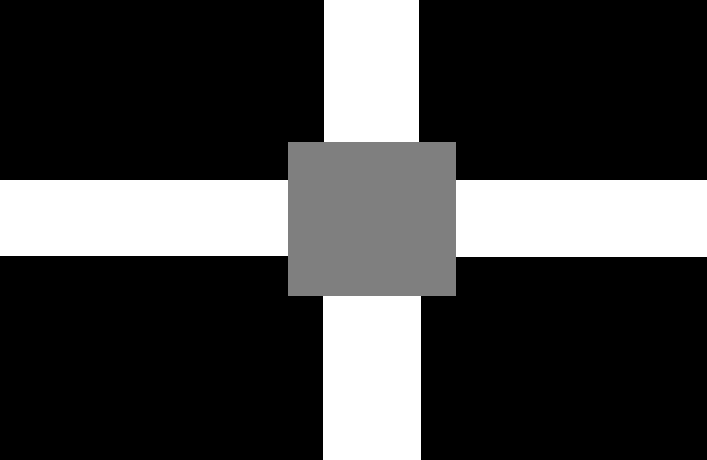}
\hspace{0.5cm}
\includegraphics[height=3.3cm,width=5cm]{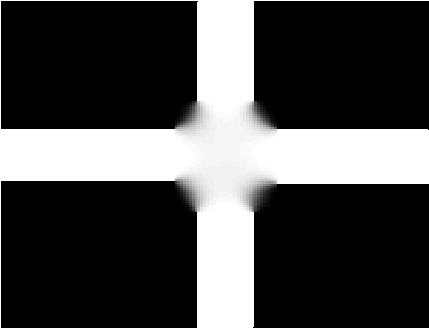}
\end{center}
\caption{Solution of inpainting problem obtained with $1000$ iterations of the ADI AMOS scheme \eqref{ADIAMOS}, $\Delta t=C(\Delta x)^3, \varepsilon=0.001$.}
\label{figinp1}
\end{figure}

As a second example, we consider a grayvalue $300\times 300$ photograph of a toucan in Figure~\ref{figinp2}. Its reconstructions is obtained in only $20$ iterations using the AMOS scheme \eqref{ADIAMOS} with $\Delta t=C(\Delta x)^3$ and $\varepsilon=0.001$.

\begin{figure}[!h]
\begin{center}
\includegraphics[height=5cm]{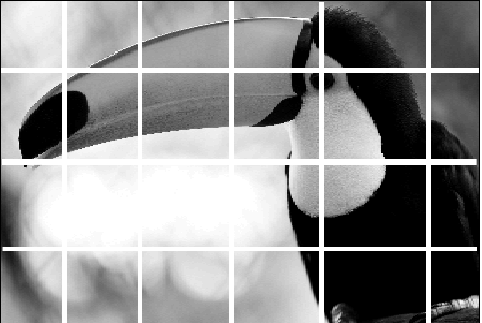}
\hspace{0.5cm}
\includegraphics[height=5cm]{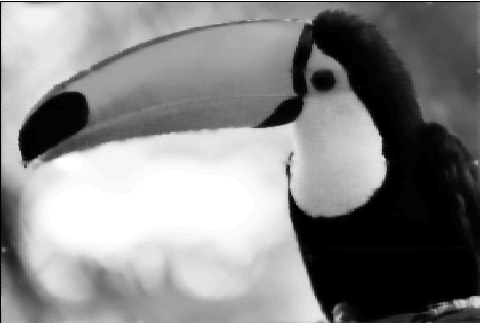}
\end{center}
\caption{Solution of the inpainting problem obtained with ADI AMOS scheme \eqref{ADIAMOS} after $20$ iterations, $\Delta t=C(\Delta x)^3, \varepsilon=0.001$.}
\label{figinp2}
\end{figure}

\section{Conclusion}
In this paper we have studied the applicability of directional
splitting techniques to fourth-order nonlinear diffusion equations. In
particular we have considered the TV-$H^{-1}$ equation which is the $H^{-1}$ gradient flow of the total
variation that emerges from imaging applications such image
denoising and inpainting.

The numerical challenges when solving this equation are both its
fourth-differential order and the strong nonlinearity
due to the total variation subgradient in the equation. In former work
on numerical schemes for this equation these issues have been either
addressed by using explicit time stepping schemes with tiny step sizes,
semi-implicit schemes with heavy damping, e.g. \cite{schon}, or by
fully implicit schemes, e.g. \cite{BuDiWei}, which are computationally
expensive to solve. Having this in mind directional splitting seems to
be a promising compromise. The computational cost of each of its
iterations is very small due to the triangular form of the system
matrices, while the stability conditions are improved when compared to
an explicit in time discretisation.

We have proposed three different ADI methods applied to two
linearisations of the TV-$H^{-1}$ equation and presented numerical results for
Gaussian-type initial conditions, image smoothing and image
inpainting. Our investigation of stability properties of the schemes
results into the following observation: any explicit terms, due to the type of splitting used, result into heavy stability conditions on the size of the time steps. For fourth-order equations these restrictions usually turn out to make a choice of $\Delta t$ as small as $(\Delta x)^4$ necessary, thus preventing an efficient numerical solution of the equation. In particular, even if each iteration of the numerical scheme is very cheap -- due to its ``one-dimensional'' character -- a large number of them have to be computed in order to see any significant change of the solution in time. The AMOS scheme provides a stable and cheap numerical scheme for solving a slightly modified version of the TV-$H^{-1}$ equation, which turns this fourth-order nonlinear equation numerically tractable and hence makes it more attractive for imaging applications.

\paragraph{Acknowledgements.}
Carola-Bibiane Sch\"{o}nlieb acknowledges the financial support provided by the Cambridge Centre for Analysis (CCA), the Royal Society International Exchanges Award IE110314 for the project \emph{High-order Compressed Sensing for Medical Imaging}, the EPSRC first grant Nr. EP/J009539/1 \emph{Sparse \& Higher-order Image Restoration}, and the EPSRC / Isaac Newton Trust Small Grant on \emph{Non-smooth geometric reconstruction for high resolution MRI imaging of fluid transport in bed reactors}. Further, this publication is based on work supported by Award No.
KUK-I1-007-43 , made by King Abdullah University of Science and Technology (KAUST).

\end{document}